\documentclass[11pt]{article}
\usepackage{amsmath}
\usepackage{amssymb}
\usepackage{amsthm}

\parskip=\medskipamount
\def\grad{\mathop{\mathrm{grad}}\nolimits}

\def\clift#1{#1^{\scriptscriptstyle{\mathrm{C}}}}
\def\hlift#1{#1^{\scriptscriptstyle{\mathrm{H}}}}
\def\hbarlift#1{#1^{\bar{\scriptscriptstyle{\mathrm{H}}}}}
\def\vlift#1{#1^{\scriptscriptstyle{\mathrm{V}}}}

\def\C{\mathcal{C}}
\def\D{\mathcal{D}}
\def\L{\mathfrak{L}}
\def\R{\mathbb{R}}
\def\S{\mathcal{S}}
\def\V{\mathcal{V}}
\def\X{\mathcal{X}}
\def\Y{\mathcal{Y}}

\def\lie#1{{\mathcal{L}}_{#1}}

\def\vectorfields#1{\mathfrak{X}(#1)}

\def\onehalf{{\textstyle\frac12}}

\def\conn#1#2#3{\setbox1=\hbox{$\scriptstyle{#2}{#3}$}%
\setbox2=\hbox to\wd1{$\hfil\scriptstyle{#1}\hfil$}
\Gamma^{\!\box2}_{\!\box1}}
\def\barconn#1#2#3{\setbox1=\hbox{$\scriptstyle{#2}{#3}$}%
\setbox2=\hbox to\wd1{$\hfil\scriptstyle{#1}\hfil$}
\check{\Gamma}^{\!\box2}_{\!\box1}}

\def\hook{{\mathchoice
{\vrule height 0pt depth 0.4pt width 3pt \vrule height 5pt depth 0.4pt
\kern 3pt}
{\vrule height 0pt depth 0.4pt width 3pt  \vrule height 5pt depth
0.4pt\kern 3pt}
{\vrule height 0pt depth 0.2pt width 1.5pt  \vrule height 3pt depth
0.2pt width 0.2pt \kern 1pt}
{\vrule height 0pt depth 0.2pt width 1.5pt  \vrule height 3pt depth
0.2pt width 0.2pt \kern 1pt} }}

\newtheorem{thm}{Theorem}
\newtheorem{lem}{Lemma}
\newtheorem{prop}{Proposition}
\newtheorem{cor}{Corollary}

\begin{document}

\title{The Cartan form for constrained Lagrangian systems and the
nonholonomic Noether theorem}

\author{M.\ Crampin and T.\ Mestdag\\
Department of Mathematics, Ghent University\\
Krijgslaan 281, S9, B--9000 Gent, Belgium}

\date{}

\maketitle

{\small {\bf Abstract.} This paper deals with conservation laws for
mechanical systems with nonholonomic constraints.  It uses a
Lagrangian formulation of nonholonomic systems and a Cartan form
approach.  We present what we believe to be the most general relations
between symmetries and first integrals.  We discuss the so-called
nonholonomic Noether theorem in terms of our formalism, and we give
applications to Riemannian submanifolds, to Lagrangians of mechanical
type, and to the determination of quadratic first integrals.\\[2mm]
{\bf Mathematics Subject Classification (2000).} 34A26, 37J60, 70G45,
70H03, 70H33.\\[2mm]
{\bf Keywords.} Lagrangian system, nonholonomic constraints, Noether
theorem, symmetries, first integrals, conservation laws.}

\section{Introduction}

This paper is concerned with some aspects of the search for
conservation laws for mechanical systems with nonholonomic constraints
(i.e.\ velocity-dependent constraints).  Such constraints show up
naturally e.g.\ in robotics and in control theory when one considers
rigid bodies rolling over a surface or over each other, or possessing
a contact point with a surface in the form of a knife edge.  Standard
references to the differential geometric approach to systems with
nonholonomic constraints are the recent books
\cite{Bloch,Cortes,Cushman}.  In this paper we follow a Lagrangian
formulation of nonholonomic systems.  The constraints will be assumed
to be linear in velocities and both the Lagrangian and the constraints
will be assumed to be independent of time.

We shall investigate the relation between first integrals (or constants
of the motion) on the one hand and symmetries on the other hand.  In
the absence of nonholonomic constraints, this relation is often
described by what is commonly called `Noether's (first)
theorem'.  Although anyone working in the field is
familiar with this terminology (see \cite{Kos} for a recent book
on the history of the subject), a quick scan through the literature
immediately reveals that the precise formulation of the theorem is
somewhat subject to personal taste.  In fact, the terminology
`Noether's theorem' is associated with many slightly different
manifestations of the same group of underlying ideas.  For example,
some authors relate Noether's theorem to invariance transformations of
the action functional (up to a gauge term) depending on position (and
possibly time) only, while others also use the term Noether's theorem
for generalizations to velocity-dependent transformations.  At the
infinitesimal level, the first point of view relates to the existence
of a vector field on the configuration space whose complete or
tangent lift to velocity space preserves the Lagrangian, while the
second viewpoint is related to the existence of a vector field on
velocity space, not necessarily projectable to one on the
configuration space, that, among other things, preserves the so-called
Cartan 2-form.  Throughout this paper, we shall refer only to the
first viewpoint as the `Noether theorem', while we shall call the
second viewpoint the `Cartan form approach'.  For unconstrained
Lagrangian systems, the relation between the two viewpoints was well
established in e.g.\ \cite{Mike}, where it is argued that the Cartan
form approach is superior since it embodies a simple direct
correspondence between symmetries and first
integrals, and in any case contains the Noether theorem (in the first
sense) as a special case.  Besides the two above-mentioned versions
of the Noether theorem, one may find in the literature many other
generalizations.  For a review on those we refer to
\cite{WillyFrans}.

Given the lack of consensus for standard Lagrangian systems, it is no
surprise that the situation is even more troublesome for the
translation of the ideas behind the Noether theorem to the context of
Lagrangian systems with additional nonholonomic constraints.  It is
well-known that for those systems the relation between symmetries and
constants of the motion is no longer as natural as it is for the
unconstrained case.  Nevertheless, many people have investigated the
conditions under which symmetry properties of the system do lead to
conservation laws.  These investigations have resulted mostly in
generalizations of the two viewpoints described above for the
unconstrained situation.  For example, papers discussing the first
point of view are those by Fass\`o and co-workers
\cite{Fasso2,Fasso4,Fasso,Fasso3} and Iliev {\em et al.\/}
\cite{Iliev,IlSem}.  Results concerning the Cartan form viewpoint may
be found in e.g.\ \cite{Gia} (for time-dependent systems) and
\cite{BatesSny,CdL,Cush} (from the Hamiltonian point of view).
Actually, the bulk of the literature focusses on a very special case:
the one where the system is invariant under the tangent lift of an
action of a Lie group (see e.g.\
\cite{BKMM,AMZ,Frans,CMR2,Koiller,Zenkov}) (none of these lists is
meant to be exhaustive).  Unfortunately, for nonholonomic systems it
is far from immediately obvious how the two distinct viewpoints are to
be compared.  The main goal of this paper is to come to a transparent
description of these two levels of generalization and of their
interaction.

Throughout the paper we shall take advantage of a simple and
well-known observation we have used also in previous publications
\cite{nonholvak,nonholsym}:\ the dynamics can be represented by means
of a vector field, and so can any symmetry (at an infinitesimal level).
In Section~2 we recall the version of the d'Alembert principle we
stated in \cite{nonholvak} and we interpret it here in terms of the
fibre metric given by the Hessian of the Lagrangian.  In the following
section we discuss some generalities concerning the restriction of the
Cartan 2-form to the constraint submanifold, we present what we
believe to be the most general relations between symmetries and
constants of the motion in the Cartan form approach and we discuss the
special case where the nonholonomic distribution is maximally
non-integrable.  In the fourth section we translate the nonholonomic
Noether theorem of \cite{Fasso} to our formalism.  We further show how
it can be derived from the results in the previous section and we
discuss some special cases.  The last section contains applications of
the previous results to Riemannian submanifolds, to Lagrangians of
mechanical type, and to the search for linear and quadratic integrals.

We shall assume that the reader is familiar with the basic tools and
concepts needed for the geometric description of Lagrangian systems,
such as the vertical and complete lifts $\vlift X$ and $\clift X$ of a
vector field $X$ on $Q$, the vertical endomorphism $S$, the concept of
a second-order ordinary differential equation vector field,
etc. For definitions and basic properties
we refer to e.g.\ \cite{CP,DR}.

\section{The Lagrange-d'Alembert principle and the fibre metric}\label{sect2}

Our starting point is the formulation of the equations determining the
dynamics of a regular Lagrangian system subject to nonholonomic linear
constraints which we gave in \cite{nonholvak}.  The constraints may be
defined by either a distribution $\D$ on configuration
space $Q$ (the constraint
distribution), or a submanifold $\C$ of $TQ$ (the constraint
submanifold).  The two are related as follows: $\C=\{u\in TQ:u\in
\D_q\subset T_qQ, q=\tau(u)\}$; $\tau$ stands here for the tangent
bundle projection $TQ\to Q$.  We assume that the dimension of each
$\D_q$ is constant and equal to $m$.  Throughout the paper we shall
denote by $\iota$ the injection $\C\to TQ$.  A vector field
$\Gamma$ on $\C$ is said to be {\em of second-order type} if it
satisfies $\tau_{*u}\Gamma=u$ for all $u\in\C$.  A Lagrangian
function $L$ on $TQ$ is said to be {\em regular with respect to $\D$}
if for any local basis $\{X_\alpha\}$ of $\D$, $1\leq\alpha\leq m$,
the symmetric $m\times m$ matrix whose entries are
$\vlift{X_\alpha}(\vlift{X_\beta}(L))$ (functions on $\C$) is
nonsingular.  In \cite{nonholvak} we proved the following proposition.

\begin{prop} \label{dAlembPrinc}
Let $L$ be a Lagrangian on $TQ$ which is regular with respect to $\D$.
Then there is a unique vector field $\Gamma$ on $\C$ which is of
second-order type, is tangent to $\C$, and is such that on $\C$
\[
\Gamma(\vlift{Z}(L))-\clift{Z}(L)=0
\]
for all $Z\in\D$.  Moreover, $\Gamma$ may be determined from the
equations
\[
\Gamma(\vlift{X_\alpha}(L))-\clift{X_\alpha}(L)=0, \quad
\alpha=1,2,\ldots m,
\]
on $\C$, where $\{X_\alpha\}$ is any local basis for $\D$.
\end{prop}

This is our version of the Lagrange-d'Alembert principle; the vector
field $\Gamma$ is the dynamical field of the constrained system.

We shall make use of the following result.  For any smooth map between
manifolds $\phi:M\to N$, if $\xi\in\vectorfields M$ (the module of
smooth vector fields on $M$) and
$\eta\in\vectorfields N$ are $\phi$-related and $\chi$ is any form on
$N$ then $\lie{\xi}(\phi^*\chi)=\phi^*(\lie{\eta}\chi)$.  We shall
apply it with $\phi=\iota$, when $\xi$ is tangent to $\C$ (so
that the restriction of $\xi$ to $\C$ is $\iota$-related to $\xi$
itself), to obtain $\lie{\xi}(\iota^*\chi)=\iota^*(\lie{\xi}\chi)$
(where on the left-hand side $\xi$ should be understood as $\xi|_\C$).
This holds, mutatis mutandis, when $\chi$ is a function.  In a similar
vein, if $\xi$ is tangent to $\C$ then
$\xi\hook\iota^*\chi=\iota^*(\xi\hook\chi)$.

We shall assume that the Lagrangian $L$, in addition to being regular
with respect to $\D$, is also regular with respect to $TQ$ (i.e.\ that
$L$ is a regular Lagrangian in the standard sense).  Remark that in
case the Hessian of $L$ is positive definite, $L$ will
automatically be regular with respect to both $TQ$ and $\D$.  The
dynamical vector field $\Gamma$ defined in the proposition above should
not be confused with the standard Euler-Lagrange field $\Gamma_0$ of
$L$ on $TQ$.  The latter is uniquely determined by the condition that
on $TQ$ (i.e.\ not only on $\C$)
\[
\Gamma_0(\vlift{Z}(L))-\clift{Z}(L)=0
\]
for all vector fields on $Q$ (i.e.\ not only those in $\D$).  We shall
not make a notational distinction between $\Gamma_0$ and its
restriction to $\C$.  It is easy to see that, on $\C$, $\Gamma -
\Gamma_0$ is vertical (with respect to the projection $\tau$).

Let $g$ be the Hessian of $L$ with respect to fibre coordinates.  It
can be considered as a fibre metric on $TQ$, so that it defines a
scalar product of vertical vectors, as follows.  Let us equip each
fibre of $TQ$, which is of course a vector space, with the flat
connection, and denote its covariant derivative operator by
$\nabla^0$.  Then for any pair of vertical vector fields $V$ and $W$,
$\nabla^0_VW$ makes sense, and is a vertical vector field.  Moreover
$\nabla^0_VW-\nabla^0_WV=[V,W]$.  Set $g(V,W)=V(W(L))-\nabla^0_VW(L)$.
Then $g$ is bilinear over $C^\infty(TQ)$ and symmetric.  If we take
$V=\vlift{X}$, $W=\vlift{Y}$ for vector fields $X$, $Y$ on $Q$ we get
back the usual definition of the Hessian,
$\vlift{X}(\vlift{Y}(L))=g(\vlift{X},\vlift{Y})$, since $\nabla^0_V\vlift{Y}=0$ for any vertical $V$; but note
that we have extended the ring of coefficients from $C^\infty(Q)$ to
$C^\infty(TQ)$.

The interesting point is that $\Gamma - \Gamma_0$ is perpendicular
(with respect to $g$) to $\vlift{X}$ for all $X$ in $\D$.  So in an
obvious sense, $\Gamma$ is the perpendicular projection of the
unconstrained dynamics $\Gamma_0$ onto $\C$.  The basic fact is
evident:\ on $\C$, $(\Gamma-\Gamma_0)(\vlift{X}(L))=0$ for all $X$ in
$\D$, since
$\Gamma(\vlift{X}(L))=\clift{X}(L))=\Gamma_0(\vlift{X}(L))$.  In view
of the definition of $g$ it follows that
$g(\Gamma-\Gamma_0,\vlift{X})=0$ for all $X\in\D$.  In fact
$g(\Gamma-\Gamma_0,V)=0$ for any vertical vector field $V$ tangent to
$\C$; or if we restrict attention to any fibre $T_qQ$, we see that
$(\Gamma-\Gamma_0)_{T_qQ}$ is normal, with respect to $g|_{T_qQ}$, to
the submanifold (indeed, linear subspace) $\C_q$ of $T_qQ$.  It will
be convenient to say that a vector field $V$ is {\em fibre-normal\/} to
$\C$ if it is vertical and $g(V,W)=0$ for all vertical vector fields
$W$ which are tangent to $\C$.  Then $\Gamma-\Gamma_0$ is fibre-normal
to $\C$.

This discussion may be summarized in the form of the following
proposition, which amounts to an alternative formulation of the
Lagrange-d'Alembert principle of Proposition~\ref{dAlembPrinc}.

\begin{prop} Let $L$ be a Lagrangian on $TQ$ which is regular with
respect to both $\D$ and $TQ$.  Then the dynamical field $\Gamma$ of
the corresponding constrained system is the unique vector field
tangent to $\C$ such that $\Gamma-\Gamma_0$ is vertical and
fibre-normal to $\C$ with respect to the fibre metric determined by
$L$.
\end{prop}

Note that $\Gamma$ is necessarily of second-order type since it
differs from a second-order differential equation field by a vertical
field.

In what follows we shall always implicitly assume that $L$ is regular
with respect to both $\D$ and $TQ$.

The fact that $\Gamma$ is the image of $\Gamma_0$ under a projection
operator appears in a paper of de Le\'{o}n and Mart\'\i n de Diego \cite{deL},
but the result is derived there in a way which somewhat obscures the role of
the Hessian $g$.

\section{The Cartan form approach}

Let $S$ be the vertical endomorphism on $TQ$.  It can be defined by
means of its action on vertical and complete lifts: $S(\clift X) =
\vlift X$ and $S(\vlift X)=0$.  It is common to call the form
$\theta_L=dL\circ S$, or more succinctly $S(dL)$, the {\em Cartan 1-form}; the {\em Cartan 2-form}
is $\omega_L = d\theta_L$.  Some authors also use `Poincar\'e-Cartan forms'
or simply `Poincar\'e forms' for these forms. According to \cite{Krupka}
Poincar\'e was the first to introduce the forms, while Cartan defined their
extension to the context of time-dependent Lagrangians. The
Cartan 2-form can be thought of as the K\"ahler lift of the Hessian
$g$.  As such, $\omega_L$ evaluated on a pair of vertical vectors
gives zero, $\omega_L(S(X),Y)=\omega_L(S(Y),X)$, and if $V$ is
vertical then $\omega_L(V,Z)=g(V,S(Z))$.  The next subsection is
concerned with the properties of the restriction of the form
$\omega_L$ to the constraint submanifold $\C$, in both senses:\
$\omega_L|_{\C}$ and $\iota^*\omega_L$.

At any point $u\in\C\subset TQ$, the tangent space $T_u\C$ projects
onto $T_qQ$ (where $q=\tau(u)$), and the kernel of the projection is
the vertical lift of $\D_q$.  We denote the vertical lift of $\D_q$ to
$u$ by $\vlift{\D}_u$, and the vertical subspace of $T_u(TQ)$ by
$V_u(TQ)$.  The subspace of $V_u(TQ)$ consisting of vectors
fibre-normal to $\vlift{\D}_u$, say $(\vlift{\D}_u)^\perp$, is
complementary to $T_u\C$ in $T_u(TQ)$.  An obvious move is to consider
those vectors $\xi\in T_u\C$ such that $S(\xi)\in\vlift{\D}_u$, and
those such that $S(\xi)\in(\vlift{\D}_u)^\perp$.  Denote the former by
$\tilde{\D}_u$, the latter by $\tilde{\D}^\perp_u$.  Then
$\tilde{\D}_u$ and $\tilde{\D}^\perp_u$ are subspaces of $T_u\C$,
which together span it; but they are not complementary, in fact
$\tilde{\D}_u\cap \tilde{\D}^\perp_u=\vlift{\D}_u$.

Let us consider the distribution $\tilde{\D}:u\mapsto\tilde{\D}_u$ in
greater detail.  It is clearly projectable to $Q$, and
$\tau_{|\C*}\tilde{\D}=\D$.  Its kernel under projection is just
$\vlift{\D}$, that is, the vertical distribution on $\C$.  So the fact
that $\tau_{|\C*}\tilde{\D}=\D$ defines $\tilde{\D}$; we can write
$\tilde{\D}_u=\tau_{|\C*u}{}^{-1}(\D_{\tau(u)})$.

We can now use the distributions $\tilde{\D}$ and $\tilde{\D}^\perp$
to construct a convenient local basis for $\vectorfields{TQ}$, or in
other words an anholonomic frame on $TQ$, adapted to the study of the
restriction of $\omega_L$ to $\C$.

\subsection{Construction of a suitable anholonomic frame}

We can evidently find local vector fields
$\X_\alpha\in\tilde{\D}$ and $\X_a\in\tilde{\D}^\perp$ such that
\begin{itemize}
\item for each $u$, $\{\tau_{*u}\X_\alpha,\tau_{*u}\X_a\}$ is a basis for $T_qQ$;
\item $\{S(\X_\alpha)\}$ is a basis for $\vlift{\D}$;
\item $\{S(\X_a)\}$ is a basis for $(\vlift{\D})^\perp$.
\end{itemize}
We could start with a basis $\{X_i\} = \{X_\alpha,X_a\} $ of vector
fields on $Q$ (an anholonomic frame on $Q$) with $\{X_\alpha\}$ a
basis for $\D$.  Take for $\X_\alpha$ some projection of
$\clift{X_\alpha}$ into $T\C$ along vertical vectors:\ then
$S(\X_\alpha)=\vlift{X_\alpha}$.  By suitable similar modifications of
the $\clift{X_a}$ we can fulfill the other requirements.  Note that
the vector fields $\X_i$ cannot be assumed to be projectable to
$Q$.

Denote $S(\X_\alpha)$ by $\Y_\alpha$ and $S(\X_a)$ by $\Y_a$.  Then
$\{\X_\alpha,\X_a,\Y_\alpha\}$ is a local basis for
$\vectorfields\C$, and $\{\X_\alpha,\X_a,\Y_\alpha,\Y_a\}$ is a
local basis for $\vectorfields{TQ}$.  We have $g(\Y_\alpha,\Y_a)=0$,
where $g$ is the fibre metric.  Let us set
$g(\Y_\alpha,\Y_\beta)=g_{\alpha\beta}$ and $g(\Y_a,\Y_b)=g_{ab}$.  If
$L$ is regular with respect to $\D$, $(g_{\alpha\beta})$ is
nonsingular.  From the above relation between $\omega_L$ and $g$, we
get
\[
\omega_L(\Y_\alpha,\X_\beta)=g_{\alpha\beta},\quad
\omega_L(\Y_a,\X_b)=g_{ab},\quad
\omega_L(\Y_\alpha,\X_a)=\omega_L(\Y_a,\X_\alpha)=0.
\]
Set $\omega_L(\X_\alpha,\X_\beta)=\omega_{\alpha\beta}$ and so on.  If
we change $\X_a$ to
$\bar{\X}_a=\X_a-g^{\alpha\beta}\omega_{a\beta}\Y_\alpha$ then
$\bar{\X}_a\in\tilde{\D}^\perp_u$, $S(\bar{\X}_a)=\Y_a$, but
\[
\omega_L(\bar{\X}_a,\X_\alpha)=\omega_{a\alpha}-
g^{\beta\gamma}\omega_{a\beta}g_{\alpha\gamma}=0.
\]
So without loss of generality we can assume that $\omega_{a\alpha}=0$.
We are still free to modify $\X_\alpha$ similarly, by the addition of
a linear combination of the $\Y_\alpha$.  Since $g_{a\alpha}=0$, this
won't alter the value of $\omega_{a\alpha}$.  Let
$\bar{\X}_\alpha=\X_\alpha-\onehalf
g^{\beta\gamma}\omega_{\alpha\gamma}\Y_\beta$.  Then
\[
\omega_L(\bar{\X}_\alpha,\bar{\X}_\beta)=\omega_{\alpha\beta}-
\onehalf g^{\gamma\delta}\omega_{\alpha\delta}g_{\beta\gamma}+
\onehalf g^{\gamma\delta}\omega_{\beta\delta}g_{\alpha\gamma}=0.
\]
So without loss of generality we can further assume that
$\omega_{\alpha\beta}=0$.

We shall assume from now on that these modifications have been made,
and we shall drop the overbars.  The constructions of the previous
paragraphs may then be summarized as follows.

Consider the following three distributions on $\C$.
\begin{itemize}
\item $\vlift{\D} = \langle \Y_\alpha \rangle$, of dimension
$\dim\D=m$, is the vertical lift of $\D$ to $\C$;
\item $\hat{\D}=\langle\X_\alpha\rangle$, of dimension
$m$, projects onto $\D$, satisfies $S(\hat{\D})=\vlift{\D}$, is
a complement to $\vlift{\D}$ in $\tilde{\D}$, and is isotropic
with respect to $\omega_L$;
\item $\tilde{\D}^\top= \langle\X_a\rangle$, of dimension $n-m$,
projects onto a complement to $\D$ in $\vectorfields Q$, is a complement to
$\tilde{\D}$ in $\vectorfields\C$, and is symplectically orthogonal to
$\tilde{\D}$; it satisfies
$S(\tilde{\D}^\top)=(\vlift{\D})^\perp$.
\end{itemize}
These three distributions are related to $\tilde{\D}$, to
$\vectorfields\C$ and to $\vectorfields{TQ}$ as follows:
\begin{align*}
\tilde{\D}&=\vlift{\D}\oplus\hat{\D} \\
\vectorfields\C&=\tilde{\D}\oplus\tilde{\D}^\top\\
\vectorfields{TQ}&=\vectorfields\C\oplus (\vlift{\D})^\perp.
\end{align*}

Let $\{\vartheta^\alpha,\vartheta^a,\varphi^\alpha,\varphi^a\}$ be the
local basis of 1-forms on $TQ$, or coframe, which is dual to the basis
$\{\X_\alpha,\X_a,\Y_\alpha,\Y_a\}$ of $\vectorfields{TQ}$.  Then
$\vartheta^i=S(\varphi^i)$ ($i=\alpha,a$).  Also,
$\{\vartheta^\alpha,\vartheta^a,\varphi^\alpha\}$, when restricted to
acting on $\vectorfields\C$, is a local basis of 1-forms on $\C$ (or
more accurately,
$\{\iota^*\vartheta^\alpha,\iota^*\vartheta^a,\iota^*\varphi^\alpha\}$
is a basis for 1-forms on $\C$; however, we shall usually ignore this
refinement, trusting that it will be clear from the context what is
intended); and $\langle\varphi^a\rangle$ is $\vectorfields\C^\circ$,
the annihilator of $\vectorfields\C$, so that $\iota^*\varphi^a=0$.
In terms of this coframe
\[
\omega_L|_\C=g_{\alpha\beta}\varphi^\alpha\wedge\vartheta^\beta
+g_{ab}\varphi^a\wedge\vartheta^b
+\onehalf\omega_{ab}\vartheta^a\wedge\vartheta^b.
\]
Evidently for $\omega_L|_\C$ to be symplectic (that is, for $L$ to be
regular) it must be the case that $(g_{ab})$ is nonsingular (this will
automatically be so if the Hessian of $L$ is positive definite).
Assuming $L$ is regular, at any $u\in\C$ the symplectic orthogonal to
$T_u\C$ in $T_u(TQ)$ is given by
\[
(T_u\C)^\top=\big\langle \X_a+g^{bc}\omega_{ab}\Y_c\big\rangle_u.
\]
The characteristic subspace $\chi(\iota^*\omega_L)$ of
$\iota^*\omega_L$ is $(T_u\C)^\top\cap T_u\C$, and is therefore given
by
\[
\chi(\iota^*\omega_L)=\{\xi^a\X_a\colon\xi^b\omega_{ab}=0\}
=\chi\left(\omega_L|_{\tilde{\D}^\top_u}\right).
\]

In summary, we have the following proposition.

\begin{prop}
We can find a frame $\{\X_\alpha,\X_a,\Y_\alpha,\Y_a\}$ on $TQ$
such that $\langle\Y_\alpha\rangle=\vlift{\D}$,
$\langle\X_\alpha,\Y_\alpha\rangle=\tilde{\D}$,
$\langle\X_\alpha,\X_a,\Y_\alpha\rangle=\vectorfields\C$,
$\langle\Y_a\rangle=(\vlift{\D})^\perp$; and such that with respect
to the dual coframe
$\{\vartheta^\alpha,\vartheta^a,\varphi^\alpha,\varphi^a\}$ we have
$\langle\varphi^a\rangle=\vectorfields\C^\circ$ and
\[
\iota^*\omega_L=g_{\alpha\beta}\varphi^\alpha\wedge\vartheta^\beta
+\onehalf\omega_{ab}\vartheta^a\wedge\vartheta^b.
\]
\end{prop}

(Strictly speaking there should be an $\iota^*$ on each 1-form on the
right-hand side of the final formula.)  We shall use this expression
for $\iota^*\omega_L$ repeatedly below.  Note that at $u\in\C$ the
first term on the right-hand side is the restriction of
$\iota^*\omega_L$ to $\tilde{\D}_u$, the second its restriction to
$\tilde{\D}^\top_u$; these are complementary subspaces of $T_u\C$ and
are symplectically orthogonal.  Moreover,
$\chi(\iota^*\omega_L)\subset\tilde{\D}^\top_u$; as is evident also
from the expression above, $\iota^*\omega_L$ is nonsingular on
$\tilde{\D}_u$.

Of course, we do not claim that the frame described in the
proposition is unique.

\subsection{First integrals and symmetries}

We are now ready to establish some relations between first integrals
on the one hand and symmetries on the other.  Some of the results
in this section are closely related to analogous results that may be
found in \cite{Gia} (albeit in the context of time-dependent
constraints) and \cite{BatesSny,Cush} (albeit in a Hamiltonian
formalism).  The reader should also keep in mind that by setting the
constraint submanifold $\C$ equal to the whole tangent manifold $TQ$,
one recovers the well-known relations between symmetries and constants
of motion, to be found in e.g.\ \cite{Mike,WillyFrans}.

Let us now work locally on $\C$, in terms of the local bases of vector
fields $\{\X_\alpha,\Y_\alpha,\X_a\}$, and the dual basis of 1-forms
(i.e.\ sections of $T^*\C\to\C$)
$\{\vartheta^\alpha,\varphi^\alpha,\vartheta^a\}$.  Note that
$\{\vartheta^a\}$ is a basis for $\tilde{\D}^\circ$, the annihilator
of $\tilde{\D}$.

Given a 1-form $\psi$ on $\C$, there need not be a vector field $Z$ on
$\C$ such that $Z\hook\iota^*\omega_L=\psi$:\ it is necessary that
$\psi(W)=0$ for all $W\in\chi(\iota^*\omega_L)$; and if
such $Z$ exists it won't be unique.  However, we have the following
important result, which establishes a modified construction for
obtaining vector fields from 1-forms.

\begin{prop}\label{corr}
For any 1-form $\psi$ on $\C$, there is a unique vector field $Z$ on
$\C$ such that $Z\in\tilde{\D}$ and $Z\hook\iota^*\omega_L-\psi\in
\tilde{\D}^\circ$.
\end{prop}

\begin{proof}
The vector field $Z$ given by
\[
Z=g^{\alpha\beta}\Big(\langle\X_\beta,\psi\rangle\Y_\alpha-\langle
\Y_\beta,\psi\rangle\X_\alpha\Big)
\]
is well defined because of the assumed regularity of $L$.  It belongs
to $\tilde{\D}$ and satisfies $Z\hook\iota^*\omega_L-\psi\in
\tilde{\D}^\circ$.  Now if $Y\in\tilde{\D}$ then
$Y\hook\iota^*\omega_L\in\langle\vartheta^\alpha,\varphi^\alpha\rangle$,
by inspection of the expression for $\iota^*\omega_L$; so if
$Y\in\tilde{\D}$ and $Y\hook\iota^*\omega_L\in\tilde{\D}^\circ$ then
$Y=0$, and the vector field $Z$ displayed above is uniquely determined.
\end{proof}

The penultimate statement is worth recording separately.

\begin{cor}\label{zero}
If $Z\in\tilde{\D}$ and $Z\hook\iota^*\omega_L\in\tilde{\D}^\circ$
then $Z=0$.
\end{cor}

It is well-known (see \cite{CP,DR}) that $\Gamma_0$, the dynamical
vector field of the unconstrained system, can be characterized as the
unique second-order vector field satisfying $\Gamma_0\hook\omega_L +
dE_L = 0$.  The function $E_L=\Delta(L)-L$ is the energy of the
Lagrangian $L$ and $\Delta$ is the Liouville field (the infintesimal
generator of scaling transformations).  It follows that the dynamical
vector field of the constrained system, $\Gamma$, is determined by the
1-form $-d(\iota^*E_L)$ in the way described in
Proposition~\ref{corr}.  Recall first of all that
$\Gamma\in\tilde{\D}$ (it is tangent to $\C$ and at each $u\in\C$,
$\tau_{*u}\Gamma=u\in\D_{\tau(u)}$); and secondly that
$\Gamma-\Gamma_0\in(\vlift{\D})^\perp$:\ say
$\Gamma=\Gamma_0+\gamma^a\Y_a$ (the $\gamma^a$ are multipliers in some
manifestation).  Then on $\C$
\[
\Gamma\hook\omega_L+dE_L=
(\Gamma-\Gamma_0)\hook\omega_L=
\gamma^a\Y_a\hook\omega_L=\gamma^bg_{ab}\vartheta^a.
\]
Now apply $\iota^*$, and the result follows.

\begin{thm}\label{int}
For $f$ a function on $\C$, let $Z_f$ be the unique vector field on
$\C$ such that $Z_f\in\tilde{\D}$ and $Z_f\hook\iota^*\omega_L-df\in
\tilde{\D}^\circ$. Then $f$ is a first integral of $\Gamma$ if and
only if $Z_f(\iota^*E_L)=0$.
\end{thm}

\begin{proof}
Since $\Gamma,Z_f\in\tilde{\D}$,
\[
\Gamma(f)=\Gamma\hook(Z_f\hook\iota^*\omega_L)=
-Z_f\hook(\Gamma\hook\iota^*\omega_L)=Z_f(\iota^*E_L),
\]
and the result follows.
\end{proof}

This procedure sets up a 1-1 correspondence between local first
integrals of $\Gamma$ (determined up to the addition of a constant)
and vector fields $Z\in\tilde{\D}$ such that
$\lie{Z}(\iota^*\omega_L)\in d(\tilde{\D}^\circ)$ (i.e.\ such that
$\lie{Z}(\iota^*\omega_L)=d\phi$ for some $\phi\in\tilde{\D}^\circ$)
and $Z(\iota^*E_L)=0$.  Indeed,
$\lie{Z}(\iota^*\omega_L)=d(Z\hook\iota^*\omega_L)=d\phi$ for
$\phi\in\tilde{\D}^\circ$ if and only if $Z\hook\iota^*\omega_L$
differs from $\phi\in\tilde{\D}^\circ$ by a closed, and so locally
exact, 1-form. If we set $\lie{Z}(\iota^*\theta_L)-\phi=dF$, then $f$
above is given by $f=F-\iota^*\theta_L(Z)=F-S(Z)(L)$.

Notice that $\iota^*E_L$ is a first integral, with
corresponding vector field $\Gamma$, though the proposition is vacuous
in this case; however, it is clear from the facts that
$\Gamma\hook\iota^*\omega_L+d(\iota^*E_L)\in\tilde{\D}^\circ$ and
$\Gamma\in\tilde{\D}$ that $\Gamma(\iota^*E_L)=0$.

We can now consider the question of whether there is any correlation
between conserved quantities and symmetries for constrained systems.
By an {\em infinitesimal symmetry\/} of $\Gamma$ we mean a
vector field $Z$ tangent to $\C$ such that $\lie{Z}\Gamma=0$.  We
shall give two different sets of conditions for deriving a symmetry
from a first integral:\ the first result involves conditions on
$\Gamma$, the second on the vector field $Z_f\in\tilde{\D}$
corresponding to a first integral $f$.  We derive them both as
corollaries of the following proposition.

\begin{prop}\label{presym}
Let $Z$ be a vector field tangent to $\C$ such that
$\tilde{\D}\hook\lie{Z}(\iota^*\omega_L)\subset\tilde{\D}^\circ$ and
$\lie{Z}(\tilde{\D})\subset\tilde{\D}$.  Let $f$ be any function on
$\C$ and $Z_f\in\tilde{\D}$ the corresponding vector field.  Then
$\lie{Z}(Z_f)\hook\iota^*\omega_L-d(Z(f))\in\tilde{\D}^\circ$. If,
further, $Z(f)=0$ then $\lie{Z}(Z_f)=0$.
\end{prop}

\begin{proof}
We have $Z_f\hook\iota^*\omega_L-df=\phi\in\tilde{\D}^\circ$.
Take the Lie derivative with respect to $Z$ to obtain
\[
\lie{Z}(Z_f)\hook\iota^*\omega_L-d(Z(f))
=-Z_f\hook\lie{Z}(\iota^*\omega_L)+\lie{Z}\phi.
\]
Since $Z_f\in\tilde{\D}$, by assumption the first term on the
right-hand side belongs to $\tilde{\D}^\circ$.
Now for any $Y\in\tilde{\D}$,
\[
\lie{Z}\phi(Y)=Z(\phi(Y))-\phi(\lie{Z}Y)=0
\]
since $\lie{Z}Y\in\tilde{\D}$ by assumption.  Thus
$\lie{Z}\phi\in\tilde{\D}^\circ$.  So
$\lie{Z}(Z_f)\hook\iota^*\omega_L-d(Z(f))\in\tilde{\D}^\circ$.  If
$Z(f)=0$ (or indeed if $Z(f)$ is constant) then
$\lie{Z}(Z_f)\in\tilde{\D}$ while
$\lie{Z}(Z_f)\hook\iota^*\omega_L\in\tilde{\D}^\circ$, and so
$\lie{Z}(Z_f)=0$ by Corollary~\ref{zero}.
\end{proof}

In the course of the proof we have in effect established that
$\lie{Z}(\tilde{\D})\subset\tilde{\D}$ if and only if
$\lie{Z}(\tilde{\D}^\circ)\subset\tilde{\D}^\circ$.

The condition
$\tilde{\D}\hook\lie{Z}(\iota^*\omega_L)\subset\tilde{\D}^\circ$ is
equivalent to $\lie{Z}(\iota^*\omega_L)(\tilde{\D},\tilde{\D})=0$, or
in other words the distribution $\tilde{\D}$ is isotropic for
$\lie{Z}(\iota^*\omega_L)$; but in view of the appeal to
Corollary~\ref{zero} the formulation in the statement of the
proposition seems preferable.

\begin{cor}\label{one}
Suppose that $\Gamma$ satisfies the conditions specified for $Z$ in
the proposition above.  Then for any first integral $f$, $Z_f$ is a
symmetry of $\Gamma$.
\end{cor}

\begin{proof}
Since $\Gamma(f)=0$, $\lie{Z_f}\Gamma=-\lie{\Gamma}(Z_f)=0$.
\end{proof}

\begin{cor}\label{other}
Let $f$ be a first integral of $\Gamma$, $Z_f$ the corresponding
vector field.  Suppose that $Z_f$ satisfies the conditions specified
for $Z$ in the proposition above.  Then $Z_f$ is a symmetry of
$\Gamma$.
\end{cor}

\begin{proof}
A certain amount of mental gymnastics is required here:\ we take
$Z_f$ for $Z$ in the proposition above, and $\Gamma$ for $Z_f$ (that
is, we take $-\iota^*E_L$ for $f$). Since $f$ is a first integral,
$Z_f(\iota^*E_L)=0$, and the final conclusion of the proposition holds.
\end{proof}

From the second of these corollaries we obtain the following theorem,
which gives sufficient conditions on a vector field $Z\in\tilde{\D}$
for it to both be a symmetry and generate a first integral.

\begin{thm}\label{sym}
Let $Z\in\tilde{\D}$ be such that
$\tilde{\D}\hook\lie{Z}(\iota^*\omega_L)\subset\tilde{\D}^\circ$,
$\lie{Z}(\tilde{\D})\subset\tilde{\D}$, $\lie{Z}(\iota^*\omega_L)\in
d(\tilde{\D}^\circ)$, and $Z(\iota^*E_L)=0$.  Then $Z$ is a symmetry of
$\Gamma$, and there is, at least locally, a function $f$ on $\C$ such
that $Z=Z_f$ and $\Gamma(f)=0$. The set of vector fields $Z$
satisfying these conditions forms a Lie algebra $\S$. For
$Z_1,Z_2\in\S$, with corresponding first integrals $f_1,f_2$, we have
$Z_1(f_2)=-Z_2(f_1)$, and the first integral corresponding to
$[Z_1,Z_2]$ is (up to an additive constant)  $Z_1(f_2)$.
\end{thm}

\begin{proof}
The vector field $Z$ satisfies the conditions of Theorem~\ref{int}
(see the remarks following it) and Corollary~\ref{other}.  If
$Z_1,Z_2\in\S$, then $Z_2\in\tilde{\D}$ and
$\lie{Z_1}(\tilde{\D})\subset\tilde{\D}$, so
$[Z_1,Z_2]=\lie{Z_1}Z_2\in\tilde{\D}$.  It is easy to see that
$[Z_1,Z_2]$ satisfies the other conditions, and also that
$k_1Z_1+k_2Z_2$, $k_1,k_2\in\R$, satisfies the conditions.  Thus $\S$
is a Lie algebra.  We have
$Z_1(f_2)=\iota^*\omega_L(Z_1,Z_2)=-Z_2(f_1)$. From
Proposition~\ref{presym} we see that
$[Z_1,Z_2]\hook\iota^*\omega_L-d(Z_1(f_2))\in\tilde{\D}^\circ$.
\end{proof}

\subsection{Systems with maximally nonintegrable constraint
distributions}

The statements in the previous sections can be further refined if one
is in the situation where the distribution is `as non-integrable as it
can be' (in a sense we shall explain next).  We shall show that one may
actually assume without loss of generality that this is always the
case.

For any distribution $\D$ on $Q$ let $[\D]$ be the smallest involutive
distribution containing $\D$; it consists of linear combinations of
repeated brackets of vector fields in $\D$, as the notation is
designed to suggest.  If $[\D]=\vectorfields{Q}$ we say that $\D$ is
{\em maximally nonintegrable}.  Other authors use other adjectives,
such as e.g.\ `totally nonholonomic' \cite{Ehlers} or `completely
nonholonomic' \cite{Mont}.

Let $\D'$ be the distribution on $TQ$ spanned by all vector fields
$\clift{X}$ and $\vlift{X}$ for $X\in\D$.
Then $\D'$ has the properties that $\tau_*\D'=\D$ and
$S(\D')=\V\cap\D'$ (where $\V$ is the vertical distribution), and it is
determined by these properties.  If $\D$ is involutive so is $\D'$,
since for any $X,Y\in\D$, $[\clift{X},\clift{Y}]=\clift{[X,Y]}\in\D'$,
$[\clift{X},\vlift{Y}]=\vlift{[X,Y]}\in\D'$, and of course
$[\vlift{X},\vlift{Y}]=0$.  Note that $\vectorfields{Q}'=
\vectorfields{TQ}$.

\begin{prop}
\[
[\D]'=[\D'].
\]
\end{prop}

\begin{proof}
Evidently $\D'\subset[\D]'$ and $[\D]'$
is involutive, so $[\D']\subset[\D]'$. On the other
hand, $[\D']$ is spanned by the repeated brackets of complete
and vertical lifts of vector fields in $\D$, and these (when nonzero)
are complete or vertical lifts of vector fields in $[\D]$:\ so
$[\D]'\subset[\D']$.
\end{proof}

\begin{cor}
If $\D$ is maximally nonintegrable so is $\D'$.
\end{cor}

Now consider, for a constrained system with constraint distribution
$\D$, those conserved quantities which are just functions on $Q$.  Of
course in the unconstrained case there aren't any; but in the
constrained case the condition for $f$ to be conserved is just that
$X(f)=0$ for all $X\in\D$.  Then evidently $X(f)=0$ for $X\in[\D]$
(see \cite{Fasso4} for a similar statement).  So these conserved
quantities are constant on the integral submanifolds of $[\D]$; and
conversely, since $\D\subset[\D]$.  So the leaves (maximal connected
integral submanifolds) of $[\D]$ are the level sets of (an independent
subset of) the conserved quantities $f$.

Let us restrict everything to a leaf $\L$ of $[\D]$.  Note
that $\D$ is still a distribution on $\L$; and since the base integral
curves of $\Gamma$ are everywhere tangent to $\D$, if they start in
$\L$ they lie in $\L$.

\begin{prop}
Let $\L$ be a leaf of $[\D]$, $\bar{L}$, $\bar{\D}$ the restrictions of $L$
and $\D$ to $\L$, then (assuming that $L$ is regular with respect to
$\D$) the dynamical field of the constrained system on $\L$ defined by
$\bar{L}$ and $\bar{\D}$ is just the restriction of $\Gamma$ to $T\L\cap\C$.
\end{prop}

\begin{proof}
We can identify $T\L$ with the leaf of $[\D']$ in $TQ$ which
projects onto $\L$, and the constraint submanifold $\bar{\C}\subset
T\L$ corresponding to $\bar{\D}$ with $T\L\cap\C$.  We know that
$\Gamma$ belongs to $\D'$, and is therefore tangent to
$\bar{\C}$.  Its restriction to $\bar{\C}$ is uniquely determined by
the restriction of the equations
$\Gamma(\vlift{X}(L))-\clift{X}(L)=0$, $X\in\D$, to $\bar{\C}$; but
these are just the Lagrange-d'Alembert equations for the system on
$\L$.
\end{proof}

So without essential loss of generality we may assume that $\D$ is
maximally nonintegrable:\ if not we just have to restrict to a leaf
$\L$ of $[\D]$.  Suppose coordinates have been chosen on $Q$ so that
the leaves of $[\D]$ are given by $x^r=\mbox{constant}$ for an
appropriate range of values of $r$:\ then these coordinates will of
course appear in the expressions for base the integral curves of the
restriction of $\Gamma$ to $\L$, but only as parameters which take the
constant values appropriate to $\L$, the leaf in which the curve lies.

Let $\tilde \D$ be as in the previous section.  We also need to
identify $[\tilde{\D}]$, the smallest involutive distribution on $\C$
containing $\tilde{\D}$.  Now for any projectable vector fields
$Z_1,Z_2$ on $\C$, $[Z_1,Z_2]$ is projectable, and
$\tau_{|\C*}[Z_1,Z_2]=[\tau_{|\C*}Z_1,\tau_{|\C*}Z_2]$.  Since
$[\tilde{\D}]$ is spanned by repeated brackets of local basis vector
fields of $\tilde{\D}$, it is projectable to $Q$, and its projection
is an involutive distribution containing $\D$.  But by construction
$\tau_{|\C*}[\tilde{\D}]$ is spanned by repeated brackets of vector
fields in $\D$, so $\tau_{|\C*}[\tilde{\D}]\subset[\D]$, whence
$\tau_{|\C*}[\tilde{\D}]=[\D]$.  Clearly the kernel of $[\tilde{\D}]$
under projection is $\vlift{\D}$.  So $[\tilde{\D}]$ is determined by
the fact that $\tau_{|\C*}[\tilde{\D}]=[\D]$; that is,
$[\tilde{\D}]_u=\tau_{|\C*u}{}^{-1}([\D]_{\tau(u)})$.  In particular,
if $\D$ is maximally nonintegrable so is $\tilde{\D}$.

Clearly if $Z\in\tilde{\D}$ is such that $\lie{Z}(\iota^*\omega_L)=0$,
so that $Z$ is a symmetry of $\iota^*\omega_L$, and
$\lie{Z}(\tilde{\D})\subset\tilde{\D}$ and $Z(\iota^*E_L)=0$, then $Z$
satisfies the conditions of Theorem~\ref{sym}. When the constraint
distribution is maximally nonintegrable we have the following partial
converses.  In the first we assume that $\D$ is 2-step maximally
nonintegrable.  A distribution $\D$ on a manifold $Q$ is {\em 2-step
maximally nonintegrable} if $\D+[\D,\D]=\vectorfields{Q}$.

\begin{prop}
If $\D$ is 2-step maximally nonintegrable then a vector field
$Z\in\tilde{\D}$ satisfying
$\tilde{\D}\hook\lie{Z}(\iota^*\omega_L)\in\tilde{\D}^\circ$,
$\lie{Z}(\tilde{\D})\subset\tilde{\D}$ and
$\lie{Z}(\iota^*\omega_L)\in d(\tilde{\D}^\circ)$ is a symmetry of
$\iota^*\omega_L$.
\end{prop}

\begin{proof}
We have $\lie{Z}(\iota^*\omega_L)=d\phi$ for some
$\phi\in\tilde{\D}^\circ$ such
that $Y\hook d\phi\in\tilde{\D}^\circ$ for all $Y\in\tilde{\D}$, or
$d\phi(Y_1,Y_2)=0$ for all $Y_1,Y_2\in\tilde{\D}$.  But for
$Y_1,Y_2\in\tilde{\D}$, $d\phi(Y_1,Y_2)=-\phi([Y_1,Y_2])$.  If $\D$ is
2-step maximally nonintegrable, so is $\tilde{\D}$.  So $\phi$
vanishes on $\vectorfields\C$, that is, $\lie{Z}(\iota^*\omega_L)=0$.
\end{proof}

For the second result we need to strengthen the first condition
of the theorem.

\begin{prop}
If $\D$ is maximally nonintegrable then a vector field
$Z\in\tilde{\D}$ satisfying
$\tilde{\D}\hook\lie{Z}(\iota^*\omega_L)=0$,
$\lie{Z}(\tilde{\D})\subset\tilde{\D}$ and
$\lie{Z}(\iota^*\omega_L)\in d(\tilde{\D}^\circ)$ is a symmetry of
$\iota^*\omega_L$.
\end{prop}

\begin{proof}
We have $\lie{Z}(\iota^*\omega_L)=d\phi$ for some
$\phi\in\tilde{\D}^\circ$ such that $Y\hook d\phi=0$ for all
$Y\in\tilde{\D}$.  We therefore consider the set $\S$ of those 1-forms
$\psi\in\tilde{\D}^\circ$ such that $Y\hook d\psi=0$ for all
$Y\in\tilde{\D}$, or equivalently $\lie{Y}\psi=0$ for all
$Y\in\tilde{\D}$.  It is an $\R$-linear subspace of
$\tilde{\D}^\circ$, and indeed a module over functions invariant under
$\tilde{\D}$.  Now if $Y$ is any vector field on $\C$ such that
$\psi(Y)=0$ and $\lie{Y}\psi=0$, and $Y'=fY$ for any function $f$ on
$\C$, then $\psi(Y')=0$, and
\[
\lie{Y'}\psi=f(Y\hook d\psi)+d(fY\hook\psi)=f\lie{Y}\psi=0.
\]
Moreover, if $Y_1,Y_2$ satisfy $\psi(Y_1)=\psi(Y_2)=0$ and
$\lie{Y_1}\psi=\lie{Y_2}\psi=0$ then
\[
\psi([Y_1,Y_2])=Y_1(\psi(Y_2))-\lie{Y_1}\psi(Y_2)=0,
\]
and
\[
\lie{[Y_1,Y_2]}\psi=[\lie{Y_1},\lie{Y_2}]\psi=0.
\]
So if $[\tilde{\D}]$ is the smallest involutive distribution
containing $\tilde{\D}$ then $\S\subset[\tilde{\D}]^\circ$.  But if
$\D$ is maximally nonintegrable then $[\tilde{\D}]=\vectorfields\C$
and so $\S=\{0\}$.  Thus $\lie{Z}(\iota^*\omega_L)=0$.
\end{proof}

\section{The nonholonomic Noether theorem}\label{nhnt}

The title of this section refers to the terminology used in the paper
\cite{Fasso} by Fass\`{o} et al.  We shall first re-express their
version of the theorem in the current framework.  Next we shall show
how it relates to the Cartan form approach.

\subsection{The theorem and the reaction-annihilator distribution}

In \cite{nonholvak} we defined (for any second-order field $\Gamma$) a
1-form $\varepsilon$ along the tangent bundle projection (restricted
to $\C$ in the nonholonomic case) by
$\varepsilon(X)=\Gamma(\vlift{X}(L))-\clift{X}(L)$ for $X$ a vector
field on $Q$.  The Lagrange-d'Alembert principle is that there is a
unique $\Gamma$ of second-order type, tangent to $\C$, such that
$\varepsilon$ annihilates $\D$.  The form $\varepsilon$ corresponding
to that particular $\Gamma$ is what Fass\`{o} et al., in \cite{Fasso},
call the reaction set $\mathcal{R}$.  The idea of Fass\`{o} et al.\ is
that there may be vector fields $Z$ on $Q$, not necessarily in $\D$,
such that $\varepsilon(Z)=0$; such a vector field belongs to the
so-called {\em reaction-annihilator distribution $\mathcal{R}^\circ$}.
The next statement is the {\em nonholonomic Noether theorem} of
\cite{Fasso2,Fasso,Fasso3}.

\begin{thm}\label{Fasso} For a vector field $Z$ on $Q$
any two of the following three conditions imply the third: (1)
$\clift{Z}(L)=0$ on $\C$; (2) $\varepsilon(Z)=0$; (3)
$\vlift{Z}(L)|_{\C}$ is a first integral of $\Gamma$.
\end{thm}

\begin{proof}
The proof is straightforward:\ we have
$\Gamma(\vlift{Z}(L))=\clift{Z}(L)+\varepsilon(Z)$, so if any two of
the terms vanish so does the third.
\end{proof}

We may equivalently express matters in terms of multipliers.
Let $\{ X_\alpha,X_a\}$ be a basis
of vector fields on $Q$, where the $X_\alpha$ span $\D$.  Evidently
$\varepsilon(X_\alpha) = 0$ on $\C$, while
$\varepsilon(X_a)=\lambda_a$, for some functions $\lambda_a$ on $\C$.
These $\lambda_a$ play the role of the Lagrangian multipliers one
finds in many formulations of the equations of nonholonomic dynamics.
Let $Z=Z^a X_a + Z^\alpha X_\alpha$; then $Z$ is in
$\mathcal{R}^\circ$ if and only if $Z^a\lambda_a =0$.  By definition,
$Z^a\lambda_a = \Gamma(Z^a \vlift{X_a}(L)) -\clift{(Z^a X_a)}(L)$.  If
now $\clift{Z}(L)=0$, then $\clift{(Z^aX_a)}(L)=-\clift{(Z^\alpha
X_\alpha)}(L)$.  Replacing this above, and taking into account the
fact that $\Gamma$ is such that $\clift{(Z^\alpha X_\alpha)}(L) =
\Gamma(Z^\alpha\vlift{X_\alpha}(L))$, we easily get that $Z^a\lambda_a
= \Gamma(\vlift{Z}(L))$, and the result follows.

In this general situation, the conserved momentum $\vlift{Z}(L)$ may
depend on the component of $Z$ transverse to $\D$.

We could introduce the following small improvement to the above
theorem:\ for any function $f$ on $Q$,
\[
\varepsilon(X)=\Gamma(\vlift{X}(L)-f)-(\clift{X}(L)-\dot{f});
\]
if any two of the three terms vanish so does the third.  If $L$ is of
mechanical type (meaning that it is of the form $T-V$, where $T$ is
associated to a Riemannian metric, and $V$ is a potential) this adds
nothing new, since if $\clift{X}(L)=\dot{f}$ then each side must be
zero, by equating to zero the separate powers of $u$.  Symmetries for
which $\clift{X}(L)=\dot{f}$ are therefore only of interest for more
general types of Lagrangians, e.g.\ for Lagrangians with magnetic
terms.  In that context, the Lagrangian is said to be {\em
quasi-invariant}, see e.g.\ \cite{Marmo}.

We shall now reinterpret the reaction-annihilator distribution
${\mathcal R}^\circ$.  It is not immediately clear which manifold
$\mathcal{R}^\circ$ is supposed to be a distribution on.  For each
point $(q,u)$ of $\C$, there is a subspace of $T_qQ$ consisting of
vectors $v$ such that $\varepsilon_{(q,u)}(v)=0$:\ but it will in
general depend on $u$.  For example, if the Lagrangian is of
mechanical type with a potential term $\phi$, then $\lambda_a$
contains two terms, one quadratic in velocities and one independent of
them:\ the first is what one gets from the kinetic energy term, the
second is just $X_a(\phi)$ (see the computations in Section~5.2).  So
the set $\{v\in T_qQ:v^a\lambda_a(q,u)=0\}$ will depend on $u$, as we
claimed above.

One may impose further conditions so that $\mathcal{R}^\circ$ can be
regarded as a distribution on $Q$.  Indeed, there may very well be
vector fields $Z$ on $Q$ such that $Z^a(q)\lambda_a(q,u)=0$ for all
$u\in\C_q$ (any $Z\in\D$ will do).  In the case of a Lagrangian of
mechanical type, since the $Z^a$ are functions of $q$ alone they must
in fact satisfy two conditions:\ the one coming from the quadratic
part of $\lambda_a$, and in addition $Z^aX_a(\phi)=0$.  It will become
clear immediately below that to restrict attention to vector fields on
$Q$ in this way (i.e.\ to think of $\mathcal{R}^\circ$ as a
distribution on $Q$) is to impose an unnecessary limitation, so we
shall not insist on it.

We can reinterpret $\varepsilon$ in terms of the fibre metric, much as
we did in Section~\ref{sect2}.  In fact for any vector field $X$ on
$Q$,
\[
\varepsilon(X)=\Gamma(\vlift{X}(L))-\clift{X}(L)=
(\Gamma-\Gamma_0)(\vlift{X}(L))=g(\Gamma-\Gamma_0,\vlift{X}).
\]
From this perspective it is clear that if we require
$\mathcal{R}^\circ$ to be in some sense a distribution, it must be
interpreted as the distribution of vector fields along $\iota$ which
are fibre-normal to $\Gamma-\Gamma_0$.  It will then consist of
vertical vector fields on $\C\subset TQ$, rather than vector fields on
$Q$; but it may contain vertical lifts of vector fields on $Q$, and it
is these which are of interest from the point of view of
Theorem~\ref{Fasso}.

\subsection{Relation to the Cartan form approach }

We now discuss the nonholonomic Noether theorem of Fass\`{o} et
al.\ from the Cartan form point of view.

We first make one further interpretation of $\varepsilon$. Notice that
$\Gamma\hook\omega_L+dE_L$ is a semi-basic 1-form along $\C$, say
$\epsilon$.  For any $v\in T_u(TQ)$ (where $u\in\C$),
$\epsilon_u(v)=g_{ab}(u)\vartheta^a(v)\gamma^b=g_u(S(v),\Gamma-\Gamma_0)$,
since $S(v)=\vartheta^\alpha(v)\Y_\alpha+\vartheta^a(v)\Y_a$.  That is
to say, $\epsilon_u(v)=\varepsilon_u(\tau_*v)$, or in other words
$\epsilon$ is $\varepsilon$ considered as a semi-basic 1-form along
$\C$. The element of $\tilde{\D}^\circ$ determined by
$\Gamma\hook\iota^*\omega_L+d(\iota^*E_L)$ according to
Proposition~\ref{corr} is $\iota^*\epsilon$.

Recall that for any vector field $X$ on $Q$, the vector field
$\clift{X}$ satisfies $\lie{\clift{X}}S=0$ and
$[\Delta,\clift{X}]=0$, where $\Delta$ is the Liouville field.

For any Lagrangian system (with Cartan forms $\theta_L$ and
$\omega_L=d\theta_L$)
\begin{align*}
\clift{X}\hook\omega_L&=\clift{X}\hook d\theta_L\\
&=\lie{\clift{X}}\theta_L-d(\clift{X}\hook\theta_L)\\
&=\lie{\clift{X}}(S(dL))-d(\vlift{X}(L))\\
&=S(d(\clift{X}(L)))-d(\vlift{X}(L)).
\end{align*}
This holds everywhere on $TQ$, and regardless of whether $L$ is regular.

Secondly,
\begin{align*}
\varepsilon(X)=\epsilon(\clift{X})
&=(\Gamma\hook\omega_L+dE_L)(\clift{X})\\
&=-(\clift{X}\hook\omega_L)(\Gamma)+\clift{X}(E_L)\\
&=-\Delta(\clift{X}(L))+\Gamma(\vlift{X}(L))+\clift{X}(\Delta(L))-\clift{X}(L)\\
&=\Gamma(\vlift{X}(L))-\clift{X}(L)
\end{align*}
as expected. This holds along $\C$, but neither $\clift{X}$ nor
$\vlift{X}$ need be tangent to $\C$.

Thirdly, we propose an analogue of Theorem~\ref{Fasso}:
\begin{thm}\label{newFasso}
For any vector field $Z$ tangent to $\C$ and for any function
$f$ on $\C$ such that $Z\hook\iota^*\omega_L-df\in \tilde{\D}^\circ$,
we have
\[
\Gamma(f)=Z(\iota^*E_L)-\iota^*\epsilon(Z);
\]
and if any two of the terms vanish so does the third.
\end{thm}
\begin{proof}
This is a small generalization of the proof of  Theorem~\ref{int}.
\end{proof}
 (Note in passing that since
$\iota^*\epsilon(Z_f)=(\Gamma\hook\iota^*\omega_L+d(\iota^*E_L))(Z_f)$
and $\Gamma\hook\iota^*\omega_L+d(\iota^*E_L)\in\tilde{\D}^\circ$,
$\iota^*\epsilon(Z_f)=0$.  Moreover, if $Z\hook\iota^*\omega_L\in
\tilde{\D}^\circ$ then evidently $\iota^*\epsilon(Z)=Z(\iota^*E_L)$.)

We shall rederive Theorem~\ref{Fasso} from the displayed formula in
the statement of Theorem~\ref{newFasso}; that is, we shall show that
Theorem~\ref{Fasso} is a special case of our analogue theorem.
Naively, one would like to substitute $\clift{X}$ for $Z$:\ but this
is not permissible since $\clift{X}$ is not necessarily tangent to
$\C$, and in any case $\clift{X}$ does not correspond directly to
$\vlift{X}(L)$ via $\omega_L$.  Let us denote by $\clift{\bar{X}}$ the
projection of $\clift{X}$ onto $\C$ along the $\Y_a$, and let us set
\[
Z=-\clift{\bar{X}}+g^{\alpha\beta}\Y_\beta(\clift{X}(L))\Y_\alpha.
\]
Then $Z$ is tangent to $\C$ and satisfies
$Z\hook\iota^*\omega_L-d(\iota^*(\vlift{X}(L)))\in\tilde{\D}^\circ$.
To see the latter, note that $S(d(\clift{X}(L)))$ is semi-basic and
$S(d(\clift{X}(L)))(\X_\alpha)=\Y_\alpha(\clift{X}(L))$, whence
$S(d(\clift{X}(L)))-\Y_\alpha(\clift{X}(L))\vartheta^\alpha
\in\langle\vartheta^a\rangle$.  Recall the formula
\[
\omega_L=g_{\alpha\beta}\varphi^\alpha\wedge\vartheta^\beta
+g_{ab}\varphi^a\wedge\vartheta^b
+\onehalf\omega_{ab}\vartheta^a\wedge\vartheta^b.
\]
We have, on $\C$,
\begin{align*}
Z\hook\omega_L&=-\clift{X}\hook\omega_L+\Y_\alpha(\clift{X}(L))\vartheta^\alpha
\pmod{\vartheta^a}\\
&=-S(d(\clift{X}(L)))+d(\vlift{X}(L))+\Y_\alpha(\clift{X}(L))\vartheta^\alpha
\pmod{\vartheta^a}\\
&=d(\vlift{X}(L))\pmod{\vartheta^a},
\end{align*}
whence $Z\hook\iota^*\omega_L-d(\iota^*(\vlift{X}(L)))\in\tilde{\D}^\circ$.
Now since $\epsilon$ is semi-basic,
$\iota^*\epsilon(Z)=-\epsilon(\clift{X})=-\varepsilon(X)$. It remains
to calculate $Z(\iota^*E_L)$. For this purpose we require the
following general result. Let $V$ be any vertical vector field. Then
from the general formula for the fibre metric $g$
\[
g(V,\Delta)=V(\Delta(L))-\nabla^0_V\Delta(L)=V(\Delta(L))-V(L)=V(E_L),
\]
since  $\nabla^0_V\Delta=V$. Thus on $\C$
\begin{align*}
Z(E_L)&=
(-\clift{\bar{X}}+g^{\alpha\beta}\Y_\beta(\clift{X}(L))\Y_\alpha)(E_L)\\
&=(-\clift{X}+\varphi^a(\clift{X})\Y_a
+g^{\alpha\beta}\Y_\beta(\clift{X}(L))\Y_\alpha)(E_L)\\
&=-\clift{X}(E_L)+\varphi^a(\clift{X})g(\Y_a,\Delta)
+g^{\alpha\beta}\Y_\beta(\clift{X}(L))g(\Y_\alpha,\Delta).
\end{align*}
Since the constraints are linear, $\Delta$ is tangent to $\C$, so
$g(\Y_a,\Delta)=0$. If we write $\Delta=\nu^\alpha\Y_\alpha$ then
$g(\Y_\alpha,\Delta)=g_{\alpha\beta}\nu^\beta$, and
\begin{align*}
g^{\alpha\beta}\Y_\beta(\clift{X}(L))g(\Y_\alpha,\Delta)&=
g^{\alpha\beta}\Y_\beta(\clift{X}(L))g_{\alpha\gamma}\nu^\gamma\\
&=\nu^\beta\Y_\beta(\clift{X}(L))=\Delta(\clift{X}(L)).
\end{align*}
It follows that on $\C$,
$Z(E_L)=-\clift{X}(E_L)+\Delta(\clift{X}(L))=\clift{X}(L)$, and
therefore $Z(\iota^*E_L)=\iota^*(\clift{X}(L))$.  So the formula
$\Gamma(f)=Z(\iota^*E_L)-\iota^*\epsilon(Z)$ becomes
$\Gamma(\iota^*(\vlift{X}(L)))=\iota^*\clift{X}(L)+\varepsilon(X)$.
This is for a particular choice of $Z$ such that
$Z\hook\iota^*\omega_L-d(\iota^*(\vlift{X}(L)))\in\tilde{\D}^\circ$.
For any other choice, say $Z'$, we have
$(Z-Z')\hook\iota^*\omega_L\in\tilde{\D}^\circ$, so that
$Z'(\iota^*E_L)-\iota^*\epsilon(Z')=Z(\iota^*E_L)-\iota^*\epsilon(Z)$,
and the same conclusion holds.

The drawback of the approach in Theorem~\ref{newFasso}, however, is
that the correspondence between first integrals and vector fields is
no longer 1-1, as it was in Theorem~\ref{int}.

\subsection{Special cases}

A particular question of interest is whether, and under what
conditions, a complete lift $\clift{X}$ can satisfy the hypotheses of
Theorem~\ref{sym} on symmetries and first integrals. For this we require that
\begin{enumerate}
\item $\clift{X}\in\tilde{\D}$,
\item $\tilde{\D}\hook\lie{\clift{X}}(\iota^*\omega_L)\subset\tilde{\D}^\circ$,
\item $\lie{\clift{X}}(\tilde{\D})\subset\tilde{\D}$,
\item $\lie{\clift{X}}(\iota^*\omega_L)\in d(\tilde{\D}^\circ)$,
\item $\clift{X}(\iota^*E_L)=0$.
\end{enumerate}
A couple of points of notation.
\begin{itemize}
\item For a distribution $\D$, we denote by $\D^1$ its first derived
distribution, which is the
distribution spanned by $\D$ and brackets of vector fields in $\D$,
that is, $\D^1=\D+[\D,\D]$. (Thus $\D$ is 2-step maximally
nonintegrable just when $\D^1=\vectorfields Q$.)
\item We denote projectable vector fields on $TQ$, and more particularly
on $\C$, with overbars; thus $\bar{Y}$ is projectable, and we set
$\tau_*\bar{Y}=Y$.  Note that $S(\bar{Y})=\vlift{Y}$.
\end{itemize}

\begin{lem}
A vector field $X$ on $Q$ is an infinitesimal symmetry of the
distribution $\D$ (that is, it satisfies $\lie{X}(\D)\subset\D$) if and
only if $\clift{X}$ is tangent to $\C$.
\end{lem}
\begin{proof}
Let $(v^a,v^\alpha)$ be the quasi-velocities corresponding to the
frame $\{X_\alpha,X_a\}$, as in \cite{nonholvak}.  Then $v^a=0$ on
$\C$.  One easily verifies that $\clift{X}(v^a) = 0$ if and only if
$[X,X_\alpha]$ is of the form $A^\beta_\alpha X_\beta$.
\end{proof}

\begin{lem}
For a vector field $X$ on $Q$, if $\clift{X}\in\tilde{\D}$ then
$\lie{\clift{X}}(\tilde{\D})\subset\tilde{\D}$.
\end{lem}

\begin{proof}
First of all, since $\clift{X}$ is evidently tangent to $\C$,
$\lie{X}(\D)\subset\D$.  We have to show that for any vector field
$Z\in\tilde{\D}$, and for any $u\in\C$, $(\lie{\clift{X}}Z)_u$ is
tangent to $\C$ and $\tau_*(\lie{\clift{X}}Z)_u\in\D_{\tau(u)}$.  It
will be enough to consider those vector fields in $\tilde{\D}$ which
are projectable.  Let $\bar{Y}$ be any projectable vector field in
$\tilde{\D}$, so that $Y\in\D$:\ then $\lie{\clift{X}}\bar{Y}$ is
tangent to $\C$ since both $\clift{X}$ and $\bar{Y}$ are; it is
projectable, and its projection $\lie{X}Y$ belongs to $\D$.
\end{proof}

So in this case  condition 3 is superfluous.

\begin{prop}
For a vector field $X$ on $Q$, $\clift{X}$ satisfies the
conditions
\begin{enumerate}
\item $\clift{X}\in\tilde{\D}$,
\item $\tilde{\D}\hook\lie{\clift{X}}(\iota^*\omega_L)\subset\tilde{\D}^\circ$,
\item $\lie{\clift{X}}(\iota^*\omega_L)\in d(\tilde{\D}^\circ)$,
\item $\clift{X}(\iota^*E_L)=0$,
\end{enumerate}
if and only if
\begin{itemize}
\item $\lie{X}(\D)\subset\D$,
\item $X\in\D$,
\item there is a (locally defined) function $F$ on $Q$ such that
$\clift{X}(L)=\dot{F}$ on $\C$,
\item for any $Y\in\D^1$, $\vlift{Y}(\clift{X}(L))=Y(F)$ on $\C$.
\end{itemize}
\end{prop}

(Notice that if $Y\in\D$, so that $\vlift{Y}$ is tangent to $\C$,
$\vlift{Y}(\clift{X}(L))=\vlift{Y}(\dot{F})=Y(F)$ on $\C$, so the
final condition is automatically satisfied.  So that condition is
really concerned with derivatives of $\clift{X}(L)$ in directions
transverse to $\C$; that is, it says something about how
$\clift{X}(L)$ changes as one moves off $\C$.  Moreover, the
transverse directions involved are those that arise from bracketing
vector fields in $\D$.)

\begin{proof}
As a preliminary step we evaluate
$\lie{\clift{X}}(\iota^*\theta_L)$.  Firstly,
$\lie{\clift{X}}\theta_L=S(d(\clift{X}(L)))$, whence
$\lie{\clift{X}}\theta_L(\bar{Y})=\vlift{Y}(\clift{X}(L))$ and
$\lie{\clift{X}}\theta_L$ vanishes on vertical vector fields (both of
these assertions holding everywhere on $TQ$).  Thus
$\lie{\clift{X}}(\iota^*\theta_L)$ vanishes on any vertical vector
field tangent to $\C$, while for any $\bar{Y}$ tangent to $\C$,
$\lie{\clift{X}}(\iota^*\theta_L)(\bar{Y})=\iota^*(\vlift{Y}(\clift{X}(L)))$.

Now suppose that $\lie{X}(\D)\subset\D$, $X\in\D$, and on $\C$,
$\clift{X}(L)=\dot{F}$ and $\vlift{Y}(\clift{X}(L))=Y(F)$ for any
$Y\in\D^1$.  We show that the numbered conditions are satisfied.

Since $\lie{X}(\D)\subset\D$, $\clift{X}$ is tangent to $\C$.  Furthermore,
$\tau|_{\C*}\clift{X}=X\in\D$, so $\clift{X}\in\tilde{\D}$, which
establishes that condition 1 is satisfied.
For $\bar{Y}\in\tilde{\D}$, so that $\vlift{Y}$ is tangent to $\C$, we
have
\[
\lie{\clift{X}}(\iota^*\theta_L)(\bar{Y})=\vlift{Y}(\iota^*(\clift{X}(L)))
=\vlift{Y}(\dot{F})=Y(F)
\]
(strictly speaking, $\tau|_\C^*(Y(F))$).  It follows that
$\lie{\clift{X}}(\iota^*\theta_L)-dF\in\tilde{\D}^\circ$ (since both
terms vanish on vertical vector fields).  Thus
$\lie{\clift{X}}(\iota^*\omega_L)\in d(\tilde{\D}^\circ)$, and
condition 3 is satisfied.  We next consider condition 2.  It can be
written $d\lie{\clift{X}}(\iota^*\theta_L)(\tilde{\D},\tilde{\D})=0$.
Using the usual formula for the exterior derivative it is easy to see
that if either or both of the arguments is vertical then one gets
zero.  For $\bar{Y}_1,\bar{Y}_2\in\tilde{\D}$,
\begin{align*}
d\lie{\clift{X}}(\iota^*\theta_L)(\bar{Y}_1,\bar{Y}_2)
&=\bar{Y}_1(\lie{\clift{X}}(\iota^*\theta_L)(\bar{Y}_2)
-\bar{Y}_2(\lie{\clift{X}}(\iota^*\theta_L)(\bar{Y}_1)\\
&\qquad-\lie{\clift{X}}(\iota^*\theta_L)([\bar{Y}_1,\bar{Y}_2])\\
&=\bar{Y}_1(Y_2(F))-\bar{Y}_2(Y_1(F))
-\lie{\clift{X}}(\iota^*\theta_L)([\bar{Y}_1,\bar{Y}_2])\\
&=[Y_1,Y_2](F)-\iota^*(\vlift{[Y_1,Y_2]}(\clift{X}(L)))\\
&=0
\end{align*}
because $\vlift{Y}(\clift{X}(L))=Y(F)$ on $\C$ for
any $Y\in\D^1$.
Finally, condition 4 follows directly from the fact that
$\Delta(\dot{F})=\dot{F}$.

For the converse we shall make use of a frame $\{X_i\}$ on $Q$ with
$\{X_\alpha\}$ a local basis for $\D$, as usual.  We take vector
fields $\bar{X}_i$ on (and tangent to) $\C$ projecting onto the
$X_i$:\ they could be the fibre-orthogonal projections onto $\C$ of
the $\clift{X_i}$, for example.  (These are not to be confused with
the $\X_\alpha$ etc., which are not necessarily projectable.)  Then
$S(\bar{X}_\alpha)=\vlift{X_\alpha}$, and
$\{\bar{X}_\alpha,\vlift{X}_\alpha\}$ is a basis for $\tilde{\D}$.
We denote by $v^i$ the corresponding quasi-coodinates; $v^a=0$ on $\C$.

Suppose that $\clift{X}$ satisfies the numbered conditions.
It follows from condition 1, firstly that $\clift{X}$ is tangent to
$\C$ and so $\lie{X}(\D)\subset\D$, and secondly that
$X=\tau_{|\C*}\clift{X}\in\D$.
Condition 3 implies the existence (locally) of a function $F$ on $\C$
such that $\lie{\clift{X}}(\iota^*\theta_L)-dF\in\tilde{\D}^\circ$.
Since $\lie{\clift{X}}(\iota^*\theta_L)$ vanishes on any vertical
vector field $V$ tangent to $\C$, and all such vector fields belong to
$\tilde{\D}$, it follows that $V(F)=0$ for all such $V$, so $F$ is
(the pull-back of) a function on $Q$.  Then
\[
X_\alpha(F)=\lie{\clift{X}}(\iota^*\theta_L)(\bar{X}_\alpha)=
\vlift{X_\alpha}(\clift{X}(L)),
\]
whence by condition 4
\[
\clift{X}(L)=\Delta(\clift{X}(L))=v^\alpha\vlift{X_\alpha}(\clift{X}(L))
=v^\alpha X_\alpha(F)=\dot{F}
\]
on $\C$. From the calculation of
$d\lie{\clift{X}}(\iota^*\theta_L)(\bar{Y}_1,\bar{Y}_2)$ above we see
that on $\C$, for any $Y\in\D^1$, $\vlift{Y}(\clift{X}(L))=Y(F)$.
\end{proof}

\begin{cor}
The numbered conditions are satisfied if $X$ is a {\em horizontal
quasi-symmetry} of the system, that is, if $\lie{X}(\D)\subset\D$,
$X\in\D$, and for some function $F$ on $Q$, $\clift{X}(L)=\dot{F}$
holds on $TQ$.  In fact $\lie{\clift{X}}(\iota^*\omega_L)=0$ in this
case.
\end{cor}

\begin{proof}
We have $\vlift{Y}(\clift{X}(L))=Y(F)$ for any vector field $Y$
on $Q$. It follows that $\lie{\clift{X}}(\iota^*\theta_L)=dF$.
\end{proof}

We may conclude, from Theorem~\ref{sym}, that if $X$ is a horizontal
quasi-symmetry then $\clift{X}$ is a symmetry of $\iota^*\omega_L$ and
of $\Gamma$, and there is, at least locally, a function $f$ on $\C$
such that $\clift{X}=Z_f$ and $\Gamma(f)=0$.  This result, with
$f=F-\iota^*(\vlift{X}(L))$ of course, is well-known; the point of the
exercise was to see how it is related to the theorem.  To turn things
around, we may say that the vector fields $Z$ satisfying the
hypotheses of Theorem~\ref{sym} should be regarded as generalizations
of horizontal quasi-symmetries.

\section{Applications}

The bulk of the literature concentrates on the case where the
Lagrangian is of mechanical type.  So we now explain how the theory
works in that special case.  A simple subcase, but one which contains
several instructive pointers to what happens in general, is furnished
by submanifolds in Riemannian geometry.

\subsection{Conservation laws in Riemannian geometry}\label{Rie}

First, some elementary remarks about geodesic conservation laws in
Riemannian geometry.

Let
$(M,g)$ be a Riemannian manifold, with Levi-Civita connection
$\nabla$.  For any vector field $Z$ on $M$, define a type $(0,2)$
tensor $K_Z$ by $K_Z(u,v)=g(\nabla_uZ,v)$.  The necessary and
sufficient condition for $Z$ to be a Killing field (infinitesimal
isometry) is that $K_Z$ is skew-symmetric; in fact
$\lie{Z}g(u,v)=K_Z(u,v)+K_Z(v,u)$.

For any curve $c$ in $M$, and any vector field $Z$,
\[
\frac{d}{dt}(g(Z,\dot{c}))=g(\nabla_{\dot{c}}Z,\dot{c})
+g(Z,\nabla_{\dot{c}}\dot{c})
=K_Z(\dot{c},\dot{c})+g(Z,\nabla_{\dot{c}}\dot{c}).
\]
If $c$ is an affinely parametrized geodesic, so that
$\nabla_{\dot{c}}\dot{c}=0$, and $K_Z(\dot{c},\dot{c})=0$, then
$g(Z,\dot{c})$ is constant along $c$. If $Z$ is a Killing field
then $g(Z,\dot{c})$ is constant along every geodesic; and
conversely (since there is a geodesic in every direction, so
$\dot{c}$ is an arbitrary vector). This, in fact, is Noether's
theorem in Riemannian geometry:\ there is a 1-1 correspondence
between geodesic invariants of the form $g(Z,\dot{c})$ and
infinitesimal symmetries, that is, isometries or Killing fields.

Now consider the case of a Riemannian submanifold $N$ of a Riemannian
manifold $(M,g)$ (the metric on $N$ is the restriction of $g$). The
second fundamental form $\Pi$ is a type $(0,2)$ tensor on $N$ with
values in the normal bundle, defined as follows. For any vector
fields $\eta$, $\zeta$ on $N$, set
$\Pi(\eta,\zeta)=\nabla_\eta\zeta^\perp$ (the normal component of
$\nabla_\eta\zeta$). For any function $f$ on $N$ we have
$\nabla_\eta f\zeta=f\nabla_\eta\zeta+\eta(f)\zeta$, and so
$\Pi(\eta,f\zeta)=f\Pi(\eta,\zeta)$, and $\Pi$ is tensorial.
Moreover, $\nabla_\eta\zeta-\nabla_\zeta\eta=[\eta,\zeta]$ and the
latter is tangent to $N$, which implies that $\Pi$ is symmetric in
its arguments.

A curve $c$ on $N$ is geodesic with respect to the induced metric if
and only if $\nabla_{\dot{c}}\dot{c}$ is normal to $N$ (here $\nabla$
is the Levi-Civita connection of $(M,g)$). If $c$ is geodesic then
$\nabla_{\dot{c}}\dot{c}=\Pi(\dot{c},\dot{c})$.

Now consider a vector field $Z$ defined in a neighbourhood of $N$
in $M$. Then for a geodesic $c$ in $N$,
\[
\frac{d}{dt}(g(Z,\dot{c}))
=K_Z(\dot{c},\dot{c})+g(Z,\nabla_{\dot{c}}\dot{c})
=K_Z(\dot{c},\dot{c})+g(Z,\Pi(\dot{c},\dot{c})).
\]
Of course if $Z$ is tangent to $N$ the last term vanishes and the
situation reduces to the one discussed earlier.  But suppose that
$Z$ is not tangent to $N$, but satisfies the following two
conditions:\ the restriction of $K_Z$ to $TN$ is skew, and $Z$ is
orthogonal to the second fundamental form of $N$.  (Of course the
codimension of $N$ must be greater than 1 for this to be possible for
$Z$ not tangent to $N$.)  Then $g(Z,\dot{c})$ is constant along
every geodesic of $N$.  In fact any two of the following conditions
implies the third:
\begin{enumerate}
\item $Z$ is orthogonal to the second fundamental form of $N$;
\item the restriction of $K_Z$ to $TN$ is skew;
\item $g(Z,\dot{c})$ is constant along every geodesic of $N$.
\end{enumerate}
The first and last of these depend only on the values of $Z$ on $N$.
Recall that $K_Z(u,v)=g(\nabla_uZ,v)$, so that $K_Z$ involves
derivatives of $Z$:\ but since we are interested only in the case
where $u$ and $v$ are tangent to $N$, we have to take derivatives only
in directions tangent to $N$, so that $K_Z(u,v)$ also depends only on
the values of $Z$ on $N$.  In other words, one could state the result
as follows:\ let $Z$ be a vector field on $N$ (but not necessarily
tangent to it --- that is, strictly speaking $Z$ is a vector field
along the injection of $N$ into $M$); then any two of the conditions
above imply the third.

Although $Z$ need not be tangent to $N$, the conserved quantity
$g(Z,\dot{c})$ depends only on its component tangent to $N$.

This result is of course the nonholonomic Noether theorem of
\cite{Fasso}, for a kinetic energy Lagrangian, in the case where the
constraints are actually holonomic!  Let us consider the
result in this light.  The equivalence between the condition
$\clift{Z}(L)=0$ on $\C$ and item 2 above is dealt with below.  The
function $\vlift{Z}(L)$ on $\C$ is more-or-less $g(Z,\dot{c})$ (the
latter is the former restricted to a base integral curve of $\Gamma$);
note that since $\Gamma$ is tangent to $\C$, in computing
$\Gamma(\vlift{Z}(L))$ we can restrict $\vlift{Z}(L)$ to $\C$, i.e.\ we
can set $v^a=0$ before acting with $\Gamma$.  (The conserved quantity
depends only on the component of $Z$ along the constraint
distribution, just as we pointed out above for the submanifold case.)
The new ingredient is the identification of the condition
$\varepsilon(Z)=0$ in terms of the second fundamental form.  In fact
the definition of the second fundamental form given above extends in a
fairly obvious way to a distribution $\D$ (assuming of course one has
a metric):\ for any vector fields $X$, $Y$ in $\D$ set
\[
\Pi(X,Y)=
\onehalf\left(\nabla_XY+\nabla_YX\right)^\perp
\]
where $Z^\perp$ is the component of $Z$ perpendicular to $\D$.
Note that symmetry is no longer automatic --- since we are now
dealing with a nonholonomic distribution we won't have
$[\eta,\zeta]\in\D$. The condition $\varepsilon(Z)=0$ is just
$Z\perp\Pi$.

\subsection{Lagrangians of mechanical type:\ the general 
case}

We shall derive expressions for $\varepsilon$ for a Lagrangian of
mechanical type,
\[
L(q,u)=\onehalf g_q(u,u)-\phi(q).
\]
The fibre metric essentially coincides with the metric $g$
on $Q$, at least so far as vertical lifts are concerned.

First of all, $\vlift{X}(L)=g(X,u)$, while
$\clift{X}(L)=g(\nabla_uX,u)-X(\phi)$.  Let us take a frame
$\{X_\alpha,X_a\}$ in which the $X_\alpha$ span $\D$ and the $X_a$ are
normal to $\D$, so that $g_{a\alpha}=0$.  As before we
denote the corresponding quasi-velocities by $(v^\alpha,v^a)$, and we
denote by $R_{\alpha\beta}^\gamma X_\gamma$ the component of
$[X_\alpha,X_\beta]$ in $\D$.  We write
$\Gamma=v^\alpha\clift{X_\alpha}+\Gamma^\alpha\vlift{X_\alpha}$.
Recall that $\Gamma$ is tangent to $\C$, and that $v^a=0$ on $\C$, so
that $u=v^\alpha X_\alpha$ on $\C$, and (for example)
$g(X_\alpha,u)=g_{\alpha\beta}v^\beta$.  Moreover,
$v^\beta\clift{X_\beta}(v^\gamma)=-R^\gamma_{\beta\delta}v^\beta
v^\delta=0$ on $\C$ (by skew-symmetry).  It follows that
\begin{align*}
\varepsilon(X_\alpha)&=
(v^\beta\clift{X_\beta}+\Gamma^\beta\vlift{X_\beta})(g(X_\alpha,u))
-g(\nabla_uX_\alpha,u)+X_\alpha(\phi)\\
&=g_{\alpha\beta}\Gamma^\beta+X_\beta(g_{\alpha\gamma})v^\beta v^\gamma
-g(\nabla_{X_\beta}X_\alpha,X_\gamma)v^\beta v^\gamma+X_\alpha(\phi)\\
&=g_{\alpha\beta}\Gamma^\beta
+g(X_\alpha,\nabla_{X_\beta}X_\gamma)v^\beta v^\gamma+X_\alpha(\phi),
\end{align*}
from which we can determine $\Gamma$. The calculation of
$\varepsilon(X_a)$ is much simplified by the choice of $X_a$, since
$g(X_a,u)=g(X_a,v^\alpha X_\alpha)=0$ on $\C$:\ thus
\[
\varepsilon(X_a)=
-g(\nabla_{X_\alpha}X_a,X_\beta)v^\alpha v^\beta+X_a(\phi)
=g(X_a,\nabla_{X_\alpha}X_\beta)v^\alpha v^\beta+X_a(\phi).
\]
In the first term in the final expression, only the component of
$\nabla_{X_\alpha}X_\beta$ normal to $\D$ matters.  Let us write
$\Pi_{\alpha\beta}=\Pi^a_{\alpha\beta}X_a$ for the symmetric part of the
normal component of $\nabla_{X_\alpha}X_\beta$ (i.e.\ the generalized
second fundamental form).  Then
\[
\varepsilon(X_a)=g_{ab}\Pi^b_{\alpha\beta}v^\alpha v^\beta +X_a(\phi),
\]
or more generally, for any vector field $Y$ on $Q$ normal to $\D$,
\[
\varepsilon(Y)=g(Y,\Pi(u,u))+Y(\phi).
\]
Since $\varepsilon$ annihilates $\D$ and $\Pi$ is normal to it, we
conclude that a vector field $Z$ on $Q$ satisfies $\varepsilon(Z)=0$
if and only if $g(Z,\Pi_{\alpha\beta})=0$ and $Z^\perp(\phi)=0$ where
$Z^\perp$ is the component of $Z$ normal to $\D$.  These are the
conditions we alluded to in the discussion on the reaction-annihilator
distribution of Section~4.1.

If $g$ is actually a flat metric (as is the case in many examples)
then the second fundamental form condition is vacuous.

The conditions for $Z$ to generate a conserved momentum are that
\begin{align*}
&g(Z,\Pi_{\alpha\beta})=0,\quad
Z^\perp(\phi)=0,\\
&g(\nabla_{X_\alpha}Z,X_\beta)+g(\nabla_{X_\beta}Z,X_\alpha)=0,\quad
Z(\phi)=0.
\end{align*}
Regarding the
potential $\phi$, note that in fact it isn't enough that $Z(\phi)=0$:\
in effect, both the component of $Z$ along $\D$ and the component
normal to $\D$ separately have to annihilate $\phi$.  The condition
$g(\nabla_{X_\alpha}Z,X_\beta)+g(\nabla_{X_\beta}Z,X_\alpha)=0$ just
says that $K_Z$, restricted to $\D$, is skew-symmetric.  This result
may also be found in e.g.\ \cite{IlSem}.

Finally, we make the link to the Riemannian case, as described in
Section~\ref{Rie}. The Euler-Lagrange field $\Gamma_0$ is the
geodesic field.  Consider a curve $c$ in $Q$, and its natural lift
$C=(c,\dot{c})$ to $TQ$.  It is easy to see (by a coordinate
calculation for example) that the tangent field to the natural lift
can be expressed as follows:
\[
\dot{C}=\Gamma_0|_C+\vlift{(\nabla_{\dot{c}}\dot{c})}.
\]
For a constrained system of this type, therefore, $c$ will be a base
integral curve of the constrained dynamical field $\Gamma$ if and only
if $\nabla_{\dot{c}}\dot{c}$ is normal to $\D|_c$.  This includes the
case in which $\D$ is integrable, giving the rule for geodesics in a
submanifold.  In fact for a base integral curve $c$ of $\Gamma$ we
have $(\Gamma-\Gamma_0)_C
=\vlift{(\nabla_{\dot{c}}\dot{c})}=\vlift{\Pi(\dot{c},\dot{c})}$. So
the formula
\[
\frac{d}{dt}(g(Z,\dot{c}))
=K_Z(\dot{c},\dot{c})+g(Z,\Pi(\dot{c},\dot{c}))
\]
(see above) is the formula
\[
\Gamma(\vlift{Z}(L))=\clift{Z}(L)+g(\Gamma-\Gamma_0,\vlift{Z})
\]
for this case (along the natural lift $C$ of $c$).

\subsection{Quadratic first integrals for constrained systems of mechanical
type}

In the previous example we have used the conditions of the
nonholonomic Noether theorem as discussed in Section~\ref{nhnt}.  In
this section we give an example where we use the more general results
of the section on the Cartan form approach.  More in particular, we
shall apply the conditions of Theorem~\ref{int}.

First, a remark about connections.  Let $E\to M$ be a vector bundle
with linear connection, with covariant derivative operator $\nabla$
and horizontal lift $X\mapsto \hlift{X}$.  Any section $\sigma$ of
$E\to M$ determines a vertical vector field on $E$ by the
identification of a point of a vector space with a (constant) vector
field on it, ie by a variant of the vertical lift construction:\ call
it therefore $\vlift{\sigma}$.  Then
$[\hlift{X},\vlift{\sigma}]=\vlift{(\nabla_X\sigma)}$.

Now let $F\to M$ be a vector sub-bundle of $E$, and suppose that $E$ is
equipped with a constant fibre metric $g$.  For any $v\in F$, let
$\bar{H}_v$ be the fibre-orthogonal projection of the horizontal
subspace $H_v$ of $T_vE$ into $T_vF$, and $X\mapsto\hbarlift{X}$ the
corresponding horizontal lift to $F$.  Then for any section $\tau$ of
$F\to M$, $[\hbarlift{X},\vlift{\tau}]$ is a vertical lift, and if we
set $[\hbarlift{X},\vlift{\tau}]=\vlift{(\bar{\nabla}_X\tau)}$ then
$\bar{\nabla}$ is the covariant derivative operator of a linear
connection on $F$.

For $x\in M$, let $\{e_i(x)\}$ be a basis for $E_x$ such that
$\{e_\alpha(x)\}$ is a basis for $F_x$ and $\{e_a(x)\}$ a basis for
its orthogonal complement. Suppose that we have, locally, such a choice of
bases depending smoothly on $x$; in other words, we have local
sections $e_i$ with these properties. Let $\{X_r\}$ be a local basis
of vector fields on $M$, and set $\nabla_{X_r}e_i=\conn jri e_j$; the
$\conn jri$ are the connection coefficients of $\nabla$ with respect
to the chosen bases. Let $(u^i)$ be the fibre coordinates on $E$
corresponding to the basis of sections $\{e_i\}$, so that
$\vlift{e_i}(u^j)=\delta^j_i$; note that $F$ is the submanifold
$u^a=0$. Then
\begin{align*}
[\hlift{X_r},\vlift{e_i}](u^j)&=
\hlift{X_r}(\delta^j_i)-\vlift{e_i}\big(\hlift{X_r}(u^j)\big)\\
&=\vlift{(\nabla_{X_r}e_i)}(u^j)=\conn jri,
\end{align*}
so that
\[
\hlift{X_r}(u^j)=-\conn jri u^i,
\]
as one would expect. Now $\hlift{X_r}-\hbarlift{X_r}\in\langle
e_a\rangle$, say $\hlift{X_r}-\hbarlift{X_r}=\xi^a_r\vlift{e_a}$.
Thus on $F$ (where $u^a=0$) $\xi^a_r=-\conn ar\alpha u^\alpha$, using
the fact that $\hbarlift{X_r}(u^a)=0$ since $\hbarlift{X_r}$ is
tangent to the submanifold $F$. It follows that
$\hbarlift{X_r}=\hlift{X_r}+\conn ar\alpha u^\alpha\vlift{e_a}$. So
finally
\[
[\hbarlift{X_r},\vlift{e_\alpha}]=
[\hlift{X_r},\vlift{e_\alpha}]+
[\conn ar\beta u^\beta\vlift{e_a},\vlift{e_\alpha}]
=\conn ir\alpha \vlift{e_i}-\conn ar\alpha \vlift{e_a}
=\conn \beta r\alpha\vlift{e_\beta};
\]
that is to say, the connection coefficients for $\bar{\nabla}$ are
$\conn \beta r\alpha$.

Now let us specialize to the case in which $E=TQ$ and $F=\C$, for a
constrained system of mechanical type.  Choose the basis of vector
fields $\{X_\alpha,X_a\}$ with $\{X_\alpha\}$ a basis for $\D$ and
$X_a$ orthogonal to $\D$; then $\vlift{X_i}$ corresponds to
$\vlift{e_i}$ above.  We take for $\nabla$ the Levi-Civita connection
of the kinetic energy metric; we have an induced connection
$\bar{\nabla}$ on sections of $\C$, i.e.\ on vector fields in $\D$,
with $\bar{\nabla}_{X_i}X_\alpha=\conn \beta i\alpha X_\beta$.  This
is all very much like the definition of the connection induced on a
submanifold in Riemannian geometry (the connection $\bar\nabla$ may in
fact be found in e.g.\ \cite{Lewis}); note that indeed the generalized
second fundamental form is given by
\[
\Pi^a_{\alpha\beta}=
\onehalf\big(\conn a\alpha\beta+\conn a\beta\alpha\big).
\]
Moreover,
\[
\Gamma=v^\alpha\clift{X_\alpha}-\Big(\conn \alpha\beta\gamma
v^\beta v^\gamma+(\grad\phi)^\alpha\Big)\vlift{X_\alpha}.
\]
We have written $\grad\phi$ for the vector field obtained by raising
the index on $d\phi$:\ that is to say, $g(X,\grad\phi)=X(\phi)$.  The
corresponding term in $\Gamma$ is (the vertical lift of) the
orthogonal projection of $\grad\phi$ onto $\D$. With our present
choice of basis this is $g^{\alpha\beta}X_\beta(\phi)X_\alpha$, where
$(g^{\alpha\beta})$ is the inverse of $(g_{\alpha\beta})$, and also
the corresponding block in the inverse of $(g_{ij})$. Likewise, the
fact that we have $\conn \alpha\beta\gamma$ in the first term in the
brackets is due to the choice of basis with $X_a$ orthogonal to $\D$,
and therefore $\vlift{X_a}$ fibre-normal to $\C$. Now on $\C$
(where $v^a=0$)
\[
\Gamma_0=v^\alpha\clift{X_\alpha}-\big(\conn i\alpha\beta
v^\alpha v^\beta+(\grad\phi)^i\big)\vlift{X_i},
\]
so that
\[
\Gamma-\Gamma_0=
\big(\Pi^a_{\alpha\beta}v^\alpha v^\beta+(\grad\phi)^a\Big)\vlift{X_a}.
\]

The vector fields
$\{\hlift{X_\alpha},\hlift{X_a},\vlift{X_\alpha},\vlift{X_a}\}$ form
a local basis for $\vectorfields{TQ}$, of course:\ denote by
$\{\theta^\alpha,\theta^a,\phi^\alpha,\phi^a\}$ the dual basis of
1-forms.
We have
\[
\hbarlift{X_i}=\hlift{X_i}+\conn ai\alpha v^\alpha\vlift{X_a};
\]
the vector fields $\{\hbarlift{X_i},\vlift{X_\alpha}\}$ form a basis
for $\vectorfields\C$, with $\{\hbarlift{X_\alpha},\vlift{X_\alpha}\}$ a basis
for $\tilde{\D}$ and $S(\hbarlift{X_a})=\vlift{X_a}$.  Notice that
since $\langle\vlift{X_a},\theta^\alpha\rangle
=\langle\vlift{X_a},\theta^b\rangle
=\langle\vlift{X_\alpha},\phi^\beta\rangle =0$,
$\{\iota^*\theta^\alpha,\iota^*\theta^a,\iota^*\phi^\alpha\}$ is the
dual basis of 1-forms on $\C$, and $\{\iota^*\theta^a\}$ is a basis
for $\tilde{\D}^\circ$. We don't claim that these bases correspond
exactly to the ones used in Section 3, but nevertheless we can use
them to analyse $\iota^*\omega_L$ in the same way as we did there.

We have
\[
\omega_L=g_{\alpha\beta}\phi^\alpha\wedge\theta^\beta+g_{ab}\phi^a\wedge\theta^b.
\]
Using the expression for $\omega_L$ given above we obtain, on $\C$,
\begin{align*}
\hbarlift{X_\alpha}\hook\omega_L&=-g_{\alpha\beta}\phi^\beta+
g_{ab}\conn a\alpha\beta v^\beta\theta^b\\
\hbarlift{X_a}\hook\omega_L&=-g_{ab}\phi^b+
g_{bc}\conn ba\alpha v^\alpha\theta^c\\
\vlift{X_\alpha}\hook\omega_L&=g_{\alpha\beta}\theta^\beta.
\end{align*}
Then for any vector
$Z=\xi^\alpha\hbarlift{X_\alpha}+\xi^a\hbarlift{X_a}+\eta^\alpha\vlift{X_\alpha}$
tangent to $\C$ we find that
\begin{align*}
Z\hook\omega_L(\hbarlift{X_\alpha})&=
g_{\alpha\beta}\eta^\beta-g_{ab}\conn a\alpha\beta v^\beta\xi^b\\
Z\hook\omega_L(\hbarlift{X_a})&=
g_{ab}\conn b\alpha\beta v^\beta\xi^\alpha+
(g_{ac}\conn cb\alpha-g_{bc}\conn ca\alpha)v^\alpha\xi^b\\
Z\hook\omega_L(\vlift{X_\alpha})&=g_{\alpha\beta}\xi^\beta.
\end{align*}
Thus the characteristic vector fields of $\omega_L$  take the form
$Z=\xi^a\hbarlift{X_a}+\eta^\alpha\vlift{X_\alpha}$ where
$\eta^\alpha=g^{\alpha\beta}g_{ab}\conn a\beta\gamma v^\gamma\xi^b$,
and $\xi^b$ satisfies
\[
(g_{ac}\conn cb\alpha-g_{bc}\conn ca\alpha)\xi^b=0.
\]
Now the connection coefficients are those for the Levi-Civita
connection, but with respect to an anholonomic frame such that
$g_{a\alpha}=0$. It follows that $g_{bc}\conn
ca\alpha+g_{\alpha\beta}\conn\beta ab=0$. Thus the condition on $\xi^a$
may be written
\[
g_{\alpha\beta}(\conn\beta ba-\conn\beta ab)\xi^b=0.
\]
But since the Levi-Civita connection has no torsion, $[X_a,X_b]=(\conn
iab-\conn iba)X_i$.  So the condition on $\xi^b$ amounts to
$[\xi^aX_a,X_b]\in\D^\perp$, or more generally
$[\xi^aX_a,\D^\perp]\in\D^\perp$.  That is, $\xi^a X_a$ (a vector
field in $\D^\perp$) has the property that its bracket with every
vector field in $\D^\perp$ belongs to $\D^\perp$.  This condition
determines a subdistribution of $\D^\perp$, which is easily seen to be
integrable (by the Jacobi identity).  It's obviously a fundamental
feature of $\D^\perp$ (or indeed of any distribution).  It is
perfectly possible, if the dimensions are right, for this distribution
to reduce to the zero vector, in which case $\iota^*\omega_L$ is
symplectic:\ this should be an interesting class of constrained
systems.  The so-called nonholonomic particle in 3 dimensions is an
example.

Now $\C$ is a vector bundle, and so has associated with it tensor
bundles of all types, which we shall call {\em $\C$-tensor bundles}.
Let $A$ be a symmetric type $(0,2)$ $\C$-tensor.  With respect to the
local basis of sections $\{X_\alpha\}$ of $\C$, $A$ has components
$A_{\alpha\beta}$; we can think of $A$ as the corresponding collection
of functions on $Q$, with the transformation rule induced by that of a
change of basis for $\D$.  Then $A_{\alpha\beta}v^\alpha v^\beta$ is a
well-defined function on $\C$.  Let $f$ be a function on $Q$.  We seek
the conditions for $\psi=A_{\alpha\beta}v^\alpha v^\beta+f$ to be a
conserved quantity for $\Gamma$.  The vector field
$Z_\psi\in\tilde{\D}$ from Theorem~\ref{int} determined by $\psi$ is
easily seen to be given by
\[
g^{\alpha\beta}\big(\hbarlift{X_\alpha}(\psi)\vlift{X_\beta}
-\vlift{X_\alpha}(\psi)\hbarlift{X_\beta}\big),
\]
so $\psi$ is conserved if and only if
\[
g^{\alpha\beta}\big(\hbarlift{X_\alpha}(\psi)\vlift{X_\beta}(E_L)
-\vlift{X_\alpha}(\psi)\hbarlift{X_\beta}(E_L)\big)=0.
\]
In the computation it helps to note that we can replace
$\hbarlift{X_\alpha}$ by $\hlift{X_\alpha}$ since there are no
occurrences of $v^a$; $\hlift{X_\alpha}(v^\beta)=-\conn
\beta\alpha\gamma v^\gamma$ on $\C$; and $\hlift{X_\alpha}(T)=0$ where
$T$ is the kinetic energy.  We obtain
\[
(X_\alpha(A_{\beta\gamma})-A_{\delta\gamma}\conn \delta\alpha\beta
-A_{\beta\delta}\conn \delta\alpha\gamma)v^\alpha v^\beta v^\gamma
+\big(X_\alpha(f)-2A_{\alpha\beta}(\grad\phi)^\beta\big) v^\alpha=0.
\]
So the necessary and sufficient conditions for
$\psi=A_{\beta\gamma}v^\beta v^\gamma+f$ to be a constant are that
\[
\bar{\nabla}_{(\alpha}A_{\beta\gamma)}=0,\quad
X_\alpha(f)=2A_{\alpha\beta}(\grad\phi)^\beta.
\]
(We have used brackets around indices to denote symmetrization, so
that for example $T_{(ijk)}$ are the components of the symmetrized
tensor derived from a tensor $T$, ie $\frac16 \sum T_{ijk}$ where the
sum is over all permutations of $i$, $j$ and $k$.  In view of the
assumption that $A_{\alpha\beta}$ is symmetric in its indices,
$\bar{\nabla}_{(\alpha}A_{\beta\gamma)}$ is in this case a constant
multiple of the cyclic sum.)

This is a direct analogue of the corresponding result for the
unconstrained system (see \cite{HS}).  That this is so is due to the
fact that we are working with an anholonomic frame, to the choice of
frame, and to the use of the induced connection $\bar{\nabla}$.

The next question is what happens if $A$ is actually the restriction
of a tensor on $Q$?  The answer is easier to state if we assume that
the tensor on $Q$ is also symmetric, but this is not essential --- all
we require is that the restriction to $\D$ is symmetric.  On the other
hand, there is no significant loss of generality in working with
symmetric tensors.  So let $A$ be a symmetric type $(0,2)$ tensor on
$Q$.  For the restriction to $\C$ of the function $A_{ij}v^iv^j+f$ to
be a conserved quantity for $\Gamma$ the components $A_{\alpha\beta}$
of $A$ must satisfy the conditions above.  Since $A$ is a tensor on
$Q$ we can attempt to express those conditions in terms of $\nabla$.
Then the condition $\bar{\nabla}_{(\alpha}A_{\beta\gamma)}=0$ becomes
\[
\nabla_{(\alpha}A_{\beta\gamma)}+A_{a(\alpha}\Pi^a_{\beta\gamma)}=0.
\]
We are not forced to take $A_{a(\alpha}\Pi^a_{\beta\gamma)}=0$,
though it is obviously convenient if it holds. The condition
$X_\alpha(f)=2A_{\alpha\beta}(\grad\phi)^\beta$ is unchanged.

Finally, suppose that $A_{ij}v^iv^j+f$ is a conserved quantity for
the unconstrained system. Under what conditions is it also a
conserved quantity for the constrained system? From the assumption
that it is conserved for the unconstrained system we have
\[
\nabla_{(i}A_{jk)}=0,\quad X_i(f)=2A_{ij}(\grad\phi)^j.
\]
From the first, in particular $\nabla_{(\alpha}A_{\beta\gamma)}=0$,
so we require in addition that $A_{a(\alpha}\Pi^a_{\beta\gamma)}=0$.
From the second, $X_\alpha(f)=2A_{\alpha\beta}(\grad\phi)^\beta+
2A_{\alpha a}(\grad\phi)^a$, so we require in addition that
$A_{\alpha a}(\grad\phi)^a=0$.

These results include and extend those of Iliev and Semerdzhiev
\cite{Iliev,IlSem} on quadratic integrals.  It is easy to extend them
also to constants of higher degree in $v$; however, this is the most
interesting case because of the interaction between $f$ and $\phi$.
For higher degree constants of the form
$A_{\alpha\beta\ldots\delta}v^\alpha v^\beta\cdots v^\delta +f$ the
conditions corresponding to the first set are
\[
\bar{\nabla}_{(\alpha}A_{\beta\gamma\ldots)}=0,\quad
X_\alpha(f)=0,\quad
A_{\alpha\beta\ldots\delta}(\grad\phi)^\delta=0.
\]

\subsubsection*{Acknowledgements}
The first author is a Guest Professor at Ghent University:\ he is
grateful to the Department of Mathematics for its hospitality.  The
second author is a Postdoctoral Fellow of the Research Foundation --
Flanders (FWO).  This work is part of the {\sc irses} project {\sc
geomech} (nr.\ 246981) within the 7th European Community Framework
Programme.  We are indebted to W.\ Sarlet for many useful discussions.


\begin{thebibliography}{99}

\bibitem{BatesSny} L.\ Bates and J.\ \'{S}niatycki, Nonholonomic
reduction, {\em Rep.\ Math.\ Phys.} {\bf 32} (1993) 99--115.

\bibitem{Bloch}
A.\,M.\ Bloch with the collaboration of J.\ Baillieul,
P.\ Crouch and J.\,E.\ Marsden, {\em Nonholonomic Mechanics and
Control\/} (Springer, 2003).

\bibitem{BKMM}
A.\,M.\ Bloch, P.\,S.\ Krishnaprasad, J.\,E.\ Marsden and
R.\,M.\ Murray, Nonholonomic mechanical systems with symmetry, {\em Arch.\
Rational Mech.\ Anal.} {\bf 136} (1996) 21--99.

\bibitem{AMZ}
A.\,M.\ Bloch, J.\,E.\ Marsden and D.\,V.\ Zenkov, Quasi-velocities
and symmetries in nonholonomic systems,  {\em Dynamical Systems\/} {\bf 24}
(2009)  187--222.

\bibitem{Frans} F.\ Cantrijn, J.\ Cort\'es, M.\ de Le\'on and D.\
Mart\'in de Diego, On the geometry of generalized Chaplygin systems,
{\em Math.\ Proc.\ Camb.\ Phil.\ Soc.} {\bf 132} (2002) 323--351.

\bibitem{CMR2} H.\ Cendra, J.\,E.\ Marsden and T.\,S.\ Ratiu,
Geometric mechanics, Lagrangian reduction, and nonholonomic systems,
in {\em Mathematics Unlimited --- 2001 and Beyond}, eds.\  B.\
Engquist and W.\ Schmid (Springer, 2001) pp.\ 221--273.

\bibitem{Cortes}
J.\ Cort\'es Monforte, {\em Geometric, Control
and Numerical Aspects of Nonholonomic Systems\/} (Lecture Notes in
Mathematics 1793, Springer, 2002).

\bibitem{CdL} J.\ Cort\'es and M.\ de Le\'{o}n, Reduction and
reconstruction of the dynamics of nonholonomic systems, {\em J.\
Phys.\ A} {\bf 32} (1999) 8615-–8645.

\bibitem{Mike} M.\ Crampin, Constants of the motion in Lagrangian
mechanics, {\em Int.\ J.\ Theor.\ Phys.} {\bf 16} (1977) 741-754.

\bibitem{HS}
M.\ Crampin, Hidden symmetries and Killing tensors, {\em Rep.\ Math.\
Phys.} {\bf 20} (1984) 31--40.

\bibitem{nonholvak}
M.\ Crampin and T.\ Mestdag, Anholonomic frames in constrained
dynamics, {\em Dynamical Systems} {\bf 25} (2010) 159--187.

\bibitem{nonholsym}
M.\ Crampin and T.\ Mestdag, Reduction of invariant constrained
systems using anholonomic frames, math.DG/1101.2551.

\bibitem{CP}
M.\ Crampin and F.\,A.\,E.\ Pirani, {\em Applicable Differential
Geometry\/} (LMS Lecture Notes 59, Cambridge University Press, 1988).

\bibitem{Cushman} R.\,H.\ Cushman, H.\ Duistermaat and J.\
\'{S}niatycki, {\em Geometry of Nonholonomically Constrained
Systems\/} (Advanced Series in Nonlinear Dynamics 26, World 
Scientific, 2010).

\bibitem{Cush} R.\ Cushman, D.\ Kemppainen, J.\ \'{S}niatycki and L.\
Bates, Geometry of nonholonomic constraints, {\em Rep.\ Math.\ Phys.}
{\bf 36} (1995) 275--286.

\bibitem{deL}
M.\ de Le\'{o}n and D.\ Martin de Diego, On the geometry of
non-holonomic Lagrangian systems, {\em J.\ Math.\ Phys.} {\bf  37} (1996)
3389--3414.

\bibitem{DR} M.\ de Le\'{o}n and P.\,R.\ Rodrigues, {\em Methods of
Differential Geometry in Analytical Mechanics\/} (North-Holland Math.\
Studies 158, Elsevier, 1989).

\bibitem{Ehlers} K.\ Ehlers, J. Koiller, R.\ Montgomery and P.\,M.\
Rios, Nonholonomic systems via moving frames:\ Cartan equivalence and
Chaplygin Hamiltonization, {\em The Breadth of Symplectic Geometry},
eds.\  J.\,E. Marsden {\em et al.} (Birkh\"{a}user, 2005) pp.\ 75--116.

\bibitem{Fasso2}
F.\ Fass\`{o}, A.\ Giacobbe and N.\ Sansonetto, Gauge conservation
laws and the momentum equation in nonholonomic mechanics, {\em Rep.\
Math.\ Phys.} {\bf 62} (2008) 345-–367.

\bibitem{Fasso4} F.\ Fass\`{o}, A.\ Giacobbe and N.\ Sansonetto, On
the number of weakly Noetherian constants of motion of nonholonomic
systems, {\em J.\ Geom.\ Mech.} {\bf 1} (2009) 389--416.

\bibitem{Fasso}
F.\ Fass\`{o}, A.\ Ramos and N.\ Sansonetto, The reaction-annihilator
distribution and the nonholonomic Noether theorem for lifted actions,
{\em Regul.\ Chaotic Dyn.} {\bf 12} (2007) 579--588.

\bibitem{Fasso3} F.\ Fass\`{o} and N.\ Sansonetto, An elemental
overview of the nonholonomic Noether theorem, {\em Int.\ J.\ Geom.\
Methods Mod.\ Phys.} {\bf 6} (2009) 1343--1355.

\bibitem{Gia}
G.\ Giachetta, First integrals of non-holonomic systems and their
generators, {\em J.\ Phys.\ A:\ Math.\ Gen.} {\bf 33} (2000) 5369--5389.

\bibitem{Iliev}
Il.\ Iliev, On first integrals of a nonholonomic mechanical system,
{\em J.\ Appl.\ Math.\ Mech.} {\bf 39} (1975) 147--150.

\bibitem{IlSem}
Il.\ Iliev and Khr.\ Semerdzdhiev, Relations between first integrals
of a nonholonomic mechanical system and of the corresponding system
freed of constraints, {\em J.\ Appl.\ Math.\ Mech.} {\bf 36} (1972) 381--388.

\bibitem{Koiller} J.\ Koiller, Reduction of some classical
non-holonomic systems with symmetry, {\em Arch.\ Rat.\ Mech.\ Anal.}\ {\bf 118}
(1992) 113--148.

\bibitem{Kos} Y.\ Kosman-Schwarzbach, {\em The Noether Theorems.
Invariance and Conservation Laws in the Twentieth Century\/} 
(Springer, 2011).

\bibitem{Krupka} D.\ Krupka, O.\ Krupkov\'a and D.\ Saunders, The
Cartan form and its generalizations in the calculus of variations,
{\em Int.\ J.\ Geom.\ Methods Mod.\ Phys.} {\bf 7} (2010) 631-654.

\bibitem{Lewis} A.\,D.\ Lewis, Affine connections and distributions with
applications to nonholonomic mechanics, {\em Rep.\ Math.\ Phys.} {\bf
42} (1998), 135--164.

\bibitem{Marmo} G.\ Marmo, G.\ Morandi, A.\ Simoni and E.\,C.\,G\
Sudarshan, Quasi-invariance and central extensions, {\em Phys.\ Rev.\
D} {\bf 37} (1988) 2196--2205.

\bibitem{Mont} R.\ Montgomery, {\em A Tour of Subriemannian
Geometries, their Geodesics and Applications} (AMS, 2002).

\bibitem{WillyFrans} W.\ Sarlet and F.\ Cantrijn, Generalizations of
Noether's theorem in classical mechanics, {\em SIAM Review} {\bf 23}
(1981) 467--494.

\bibitem{Zenkov}
D.\,V.\ Zenkov, Linear conservation laws of nonholonomic systems with
symmetry, {\em Proceedings of the Fourth International Conference on
Dynamical Systems and Differential Equations,} Wilmington, NC, 2002,
pp.\ 967--976.

\end{thebibliography}
\end{document}